\documentclass{article}

\usepackage{comment}
\usepackage{amsmath}
\usepackage{amsthm}
\usepackage{amssymb}
\usepackage{mathtools}
\usepackage[utf8]{inputenc}

\newtheorem{theorem}{Theorem}
\newtheorem{thm}{Theorem}[section]
\newtheorem{proposition}[thm]{Proposition}
\newtheorem*{proposition*}{Proposition}
\newtheorem{corollary}[thm]{Corollary}
\newtheorem{conjecture}{Conjecture}
\newtheorem{lemma}[thm]{Lemma}

\theoremstyle{definition}
\newtheorem{definition}{Definition}[section]
\newtheorem{remark}{Remark}[section]

\newtheorem{example}{Example}
\newtheorem*{example*}{Example}

\usepackage{titlesec}
\titlespacing*{\section}
{0pt}{5.5ex plus 1ex minus .2ex}{4.3ex plus .2ex}
\titlespacing*{\subsection}
{0pt}{5.5ex plus 1ex minus .2ex}{4.3ex plus .2ex}

\usepackage{import}
\usepackage{pdfpages}
\usepackage{transparent}
\usepackage{xifthen}

\usepackage{hyperref}
\usepackage{mathtools}
\usepackage{tikz-cd}
\usetikzlibrary{intersections}
\usepackage{makeidx}
\usepackage{mathrsfs}
\usepackage{array}

\usepackage{afterpage}

\def\zn{,\kern-0.09em,}

\begin{document}

\begin{center}
    \textbf{\large Riemannian distance and symplectic embeddings in cotangent bundle}
\end{center}

\begin{center}
    Filip Broćić\\
    \textit{Department of Mathematics and Statistics,}\\
    \textit{University of Montr\'eal, H3T 1J4 Montr\'eal, Canada}\\
    \textit{filip.brocic@umontreal.ca}
\end{center}

\begin{abstract}	
 Given an open neighborhood $\mathcal{W}$ of the zero section in the cotangent bundle of $N$ we define a distance-like function $\rho_\mathcal{W}$ on $N$  using certain symplectic embeddings from the standard ball $B^{2n}(r)$ to $\mathcal{W}$. We show that when $\mathcal{W}$ is the unit disc-cotangent bundle of a Riemannian metric on $N$, $\rho_\mathcal{W}$ recovers the metric. As an intermediate step, we give a new construction of a symplectic embedding of the ball of capacity 4 to the product of Lagrangian discs $P_L := B^n(1)\times B^n(1)$, and we give a new proof of the strong Viterbo conjecture about normalized capacities for $P_L$. We also give bounds of the symplectic packing number of two balls in a unit disc-cotangent bundle relative to the zero section $N$.
\end{abstract}

\section{Introduction}

Fix a symplectic manifold $(M,\omega)$ and two Lagrangian submanifolds  $L_1$ and $L_2$. Assume that $L_1$ and $L_2$ intersect transversely at a single point  $p \in L_1 \cap L_2$. One way to estimate quantitatively how $L_1$ and $L_2$ intersect at $p$ is to consider the supremum of $\pi r^2$ over all symplectic embeddings $e:B^{2n}(r) \to M$ of a standard ball, with $e(0)=p$, such that $e$ maps the real part of $B^{2n}(r)$ to $L_1$ and the imaginary part to $L_2$. Denote this number by $Gr(L_1,L_2; M)$. This invariant is a version of the Gromov width, relative to Lagrangians, and has appeared before in \cite{Le07, BCS21}.  Here and in the rest of the paper, real and imaginary parts are with respect to the identification $\mathbb{R}^{2n}=\mathbb{C}^n=\mathbb{R}^n \times \mathbb{R}^n$. 

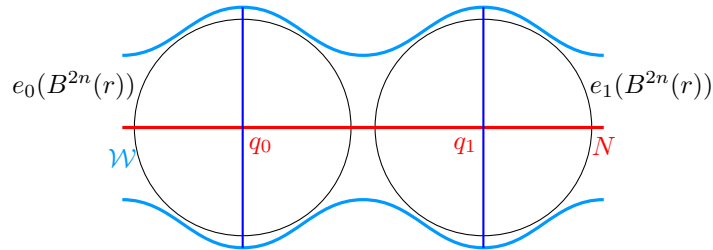
\begin{figure}[h!]
 \centering
   \begin{tikzpicture}[scale=0.80]
\draw[color=blue!40!cyan, very thick] (-4,1.2)  to [out=0, in=180] (-2, 2) to [out=0, in=180] (0,1.2) to [out=0, in=180] (2, 2) to [out=0, in=180] (4,1.2);
\draw[color=blue!40!cyan, very thick] (-4,-1.2)  to [out=0, in=180] (-2, -2) to [out=0, in=180] (0,-1.2) to [out=0, in=180] (2, -2) to [out=0, in=180] (4,-1.2);
\draw (-2,0) circle (1.8);
\draw (2,0) circle (1.8);
\node[red] at (-1.7,-0.3) {$q_0$};
\node[red] at (1.7,-0.3) {$q_1$};
\node at (-4.8,0.7) {$e_0(B^{2n} (r))$};
\node at (4.8,0.7) {$e_1(B^{2n} (r))$};
\node[red] at (4,-0.3) {$N$};
\node[color=blue!40!cyan] at (-4,-0.5) {$\mathcal{W}$};
\draw[red, very thick] (-4,0) -- (4,0);
\draw[blue, thick] (-2,-2) -- (-2,2);
\draw[blue, thick] (2,-2) -- (2,2);
\end{tikzpicture}
    \caption{Symplectic embeddings with disjoint images, centered at $q_0$ and $q_1$, with real parts mapped to zero section $N$ and imaginary parts mapped to fibers $T^*_{q_i}N$. }
    \label{Functionrho}
\end{figure}
 
 One example of such a pair is when $L_1 = N$ is the zero section of $T^*N$ and $L_2 = T^*_q N$ is a fiber over a point $q\in N$. In this case, the invariant $Gr(N,T^*_q N; T^*N)$ is infinite. Given an open and bounded neighborhood $\mathcal{W}\subset T^*N$ of the zero section $N$, and $\mathcal{W}_q:=T^*_qN \cap \mathcal{W}$, then $Gr(N,\mathcal{W}_q; \mathcal{W})$ becomes finite with a trivial upper bound coming from the volume obstruction. 
 	
	Using embeddings with constraints on the real and imaginary part, one can measure how two points $q_0 \neq q_1 \in N$ are separated inside $\mathcal{W}$, in a symplectic sense. To this aim we introduce a quantity $\rho_{\mathcal{W}}(q_0, q_1)$ as the supremum of $\pi r^2/2$ over symplectic embeddings $e_i:B^{2n}(r) \to \mathcal{W}$ of two standard balls, with disjoint images, centered at $q_0$ and  $q_1$ respectively (see Figure \ref{Functionrho}). We extend the function $\rho_{\mathcal{W}}$ to $N\times N$ by assigning the value $0$ when $q_0=q_1$. See Definition \ref{MainDefinition} for a precise description of $\rho_{\mathcal{W}}$. 
 
 Following an idea from metric geometry (see \cite{BBI01}), we define a length of a curve $\gamma:[a,b] \rightarrow N$   associated with the function $\rho_\mathcal{W}$.

\begin{definition}
For any piece-wise smooth curve $\gamma:[a,b] \rightarrow N$ we define it's length with respect to $\rho_{\mathcal{W}}$ as
$$
L_{\rho_{\mathcal{W}}}(\gamma) = \sup_{\mathcal{P}} \sum_{1\leq i \leq k}\rho_\mathcal{W}(\gamma(t_i) , \gamma(t_{i+1}) ),
$$
where $\mathcal{P}$ is a partition of the segment $[a,b]$ given by $a=t_1 < t_2 < \cdots < t_{k+1} = b$.
\end{definition}

The main result of the paper relates the length structure $L_{\rho_\mathcal{W}}$ to the standard length structure coming from a Riemannian metric $g$ in the case when the neighborhood $\mathcal{W}\subset T^*N$ is the unit-disc bundle $D_g^*N:= \{p \in T^*N \mid \|p\|_g < 1 \}.$We will write $D^*N$ when $g$ is clear from the context.

\begin{theorem}\label{Main}
Let $(N,g)$ be a closed Riemannian manifold. If $\mathcal{W}=D^*N$ then
$$L_{\rho_{\mathcal{W}}}(\gamma) = \int_a^b \|\gamma'(t)\|_g dt.$$
\end{theorem}

The proof of the theorem is given in Section \ref{ProofMain}. The idea of the proof goes as follows. Set $L_g (\gamma) : = \int_a^b \|\gamma'(t)\|_g dt$. The inequality  $L_{\rho_{\mathcal{W}}} \geq L_g$ follows from the elementary construction of a symplectic embedding $e:B^{2n}(r) \to D^*N$ with the following properties. The image of the embedding $e$ is contained in the restriction of $D^*N$ to the normal neighborhood around $q\in N$,  of radius $d$. It turns out that the capacity of the ball $B^{2n}(r)$ can be very close to $d$, for $d$ small enough. We prove that the existence of such an embedding $e$ implies that $\rho_\mathcal{W}(q_0, q_1) \geq d_g(q_0, q_1) - C(g) d_g^2(q_0, q_1)$, for $d_g(q_0, q_1)$ small enough and some constant $C(g)>0$. The construction of $e$ relies on the fact that the derivative of the exponential map $exp:T_qN \to N$ is ``close" to the identity $Id:T_qN \to T_qN$ near $0 \in T_qN$. These observations are the content of Lemmas \ref{Sharpness for bi-disc} and \ref{Local distance}. The other inequality, $L_{\rho_{\mathcal{W}}} \leq L_g$ is not elementary, and it follows from the Proposition \ref{MainProp}, where we show that $\rho_\mathcal{W}(q_0, q_1) \leq d_g(q_0, q_1)$, when $\mathcal{W}=D^*N$.  We prove this proposition in Section \ref{MainPropProof}. In the proof, we use a Floer-type theory (\cite{Fl88}), called wrapped Floer homology. It was introduced by Abbondandolo and Schwarz in \cite{ASch05} for fibers in $T^*N$ and extended by Abouzaid and Seidel in \cite{AS10} to more general Lagrangians inside any Liouville manifold $M$. Here, we use the isomorphisms from \cite{ASch05, Ab12} between the wrapped Floer homology $HW(T^*_{q_0}N, T^*_{q_1}N)$ of two fibers in $T^*N$ and the Morse homology $HM(\mathcal{P}(q_0,q_1))$ of the space of paths on $N$ between $q_0$ and $q_1$. From the properties of these isomorphisms, we extract a $J$-holomorphic curve $u$ with special boundary conditions. The energy $E(u)$ of such $u$ is bounded from above by $d_g(q_0,q_1)$. The inequality $\rho_\mathcal{W}(q_0, q_1) \leq d_g(q_0,q_1)$ follows from the monotonicity theorem for minimal surfaces. 

\begin{remark}
\begin{itemize}
    \item[(i)] The functional $L_{\rho_{\mathcal{W}}}$ is defined using only the symplectic structure. Thus, our result shows that the symplectic structure can recover the Riemannian length, and hence the Riemannian distance. 
    \item[(ii)] Even if the boundary of the neighborhood $\mathcal{W}$ is smooth, the associated length structure doesn't need to be smooth. We give examples of this phenomenon in Figure \ref{NbhdExmp}. One example is of a neighborhood $\mathcal{W}$ of $S^1=\mathbb{R} / \mathbb{Z}$ with smooth boundary, such that associated length structure $L_{\rho_{\mathcal{W}}}$ is not continuous (with respect to $C^1$ topology on the space of paths). The second example is a neighborhood $\mathcal{W}$, which is a unit-disc bundle of a smooth Finsler metric on $S^1$, but the associated length structure fails to be differentiable. One could construct similar examples in higher dimensions.
    
    \item[(iii)] The function $\rho_\mathcal{W}$ is not a metric in general since it may not satisfy the triangle inequality. One can associate to $\rho_{\mathcal{W}}$ a pseudo-metric $$D_\mathcal{W}(q_0, q_1):= \inf \left\{ \sum_{1\leq i \leq k} \rho_\mathcal{W}(x_i, x_{i+1}) \mid x_i \in N, x_1=q_0, x_{k+1}=q_1\right\}.$$
    From Proposition \ref{Equivalence of metrics}, it follows that $D_\mathcal{W}$ is non-degenerate, hence a metric. For a fixed Riemannian metric $g$ on $N$, $D_\mathcal{W}$ is equivalent to the Riemannian distance $d_g$ associated with $g$. Proposition \ref{Equivalence of metrics} remains true when one replaces $\rho_\mathcal{W}$ with $D_\mathcal{W}$. From the proof of the Theorem \ref{Main}, one can see that length functional associated with the metric $D_\mathcal{W}$ is equal to $L_{\rho_\mathcal{W}}$.

\end{itemize}

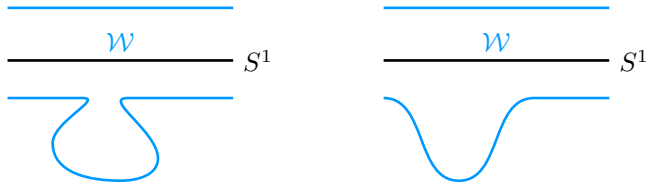
\begin{figure}[ht]
    \centering
    \begin{tikzpicture}
    \begin{scope}
      \node[blue!40!cyan] at (1.5,0) [above] {$\mathcal{W}$};
      \draw[line width=1pt] (0,0) -- (3,0) node [right] {$S^{1}$};
      \draw[blue!40!cyan,line width=1pt] (0,0.7)--(3,0.7) (0,-0.5) to[out=0,in=180] (1,-1.6) to[out=0,in=180] (2,-0.5) to[out=0,in=180] (3,-0.5);
    \end{scope}
    \begin{scope}[shift={(-5,0)}]
      \node[blue!40!cyan] at (1.5,0) [above] {$\mathcal{W}$};
      \draw[line width=1pt] (0,0) -- (3,0) node [right] {$S^{1}$};
      \draw[blue!40!cyan,line width=1pt] (0,0.7)--(3,0.7) (0,-0.5) -- (1,-0.5) to[out=0,in=90] (0.6,-1.1) to[out=-90,in=180] (1.5,-1.6) to[out=0,in=-90](2,-1.3) to[out=90,in=180] (1.6,-0.5)--(3,-0.5);
    \end{scope}
  \end{tikzpicture}
    \caption{On the left: a neighborhood $\mathcal{W}$ with the smooth boundary, but $L_{\rho_\mathcal{W}}$ is not $C^0$. On the right: a  neighborhood $\mathcal{W}$, which is the unit-disc bundle of a smooth Finsler metric, but $L_{\rho_\mathcal{W}}$ is not $C^1$.}
    \label{NbhdExmp}
\end{figure}
\end{remark}

\subsection{Other results}

 Given a Lagrangian submanifold $L$ of $M$, we say that a symplectic embedding $e:B^{2n}(r)\rightarrow M$ is \textit{relative} to $L$ if  the preimage  $e^{-1}(L) = B^n(r) \times \{0\}$ is equal to the real part of the ball $B^{2n}(r)$. The main ingredient for the proof of the Theorem \ref{Main} appears in the next proposition. It gives a bound on the radii of two symplectically embedded standard balls relative to the zero section $N$ in a unit disc-cotangent bundle $D^*N$. We also assume that the embedding $e:B^{2n}(r_0) \sqcup B^{2n}(r_1) \to D^*N$ satisfies the following constraint on the imaginary parts
\begin{equation}\label{ImPartCond}
     e^{-1}(D^*_{q_0}N \sqcup D^*_{q_1} N) = \{0\} \times B^{n}(r_0) \sqcup \{0\} \times B^{n}(r_1).
\end{equation}

\begin{proposition}\label{MainProp}
Let $(N,g)$ be a closed, connected, Riemannian manifold. If $e:B^{2n}(r_0) \sqcup B^{2n}(r_1) \to D^*N$ is a symplectic embedding relative to $N$ which satisfies the condition (\ref{ImPartCond}) then $\pi r_0^2 + \pi r_1^2 \leq 4 d_g(q_0, q_1).$
\end{proposition}

  As a corollary, we get a bound on the symplectic packing numbers for two balls relative to the zero section $N \subset D^*N$. Also, we get a bound on the capacity of one ball, symplectically embedded to $D^*N$ in the complement of a fiber $D^*_qN$, relative to the zero section $N$.

\begin{corollary} \label{RelGr}  $\widetilde{Gr}^2(N, D^*N) \leq 4 diam(N).$
\end{corollary} 

\begin{corollary} \label{RelGr2}  $\widetilde{Gr}(N\setminus \{q\}, D^*(N\setminus \{q\})) \leq 4 diam(N).$
\end{corollary}
$\widetilde{Gr}(N, D^*N)$ is a version of Gromov width relative to Lagrangians. It is similar to the invariant $Gr(L_1, L_2; M)$ from the previous section. A precise definition appears in Section \ref{RelGrW}. When $N$ is a sphere of revolution the full Gromov width of $D^*N$ was calculated in \cite{FRV23}. Their result provides a better upper bound for  $\widetilde{Gr}(N; D^*N)$ than Corollary \ref{RelGr2} for such $N$.
 
Another corollary of Proposition \ref{MainProp} is a bound on the relative Gromov width of the product of Lagrangian discs. Set $D^n_q(a):= \{(q,0) \in  \mathbb{R}^n \times \mathbb{R}^n \mid \|q\| < a\},$  and $D^n_p (b) := \{(0,p) \in \mathbb{R}^n \times \mathbb{R}^n \mid \|p\|< b \}.$ We will consider the product of such discs as a symplectic manifold and denote it by
$P^{2n}_L(a,b) := D^n_q(a) \times D^n_p(b).$
In \cite{Ra17} Ramos showed that  $Gr(P^4_L(1,1))=4.$ Using the notation from \cite{Ra17}, we call $P^{2n}_L(a,b)$ a Lagrangian bi-disc. One can think of $P^{2n}_L(a,b)$ as a local model for $D^*N$.
\begin{corollary}\label{Capacity of a Lagrangian bi-disc}
The Gromov width of a bi-disc $P^{2n}_L(a,b)$ relative to $D^n_q(a)$ and $D^n_p(b)$ satisfies
$$
Gr(D^n_q(a), D^n_p(b); P^{2n}_L(a,b)) = 4ab.
$$
\end{corollary}
\begin{proof}
It is easy to see that $P_L(a,b)$ is symplectomorphic to $P_L(ab,1)$. Let $T^n:= \mathbb{R}^n / ((2ab+1) \mathbb{Z})^n$ be a flat torus. There is an obvious symplectic embedding $i:P_L(ab,1) \to D^*T^n$. Fix a relative symplectic embedding $e:B^{2n}(r) \to P_L,$ and set $q_0:=e(0).$ For every $\epsilon>0$ there exists a point $q_1 \in T^n$ such that $d(q_0,q_1)\leq ab+\epsilon $.
From the Proposition \ref{MainProp} we know that $\pi r^2 \leq 4d(q_0,q_1) \leq 4ab+4\epsilon.$ Since $\epsilon>0$ was arbitrary we get
$$
Gr(D^n_q(a), D^n_p(b); P^{2n}_L(a,b))\leq 4ab.
$$
We construct an explicit relative symplectic embedding from a ball of capacity $4ab$ to $P_L(a,b)$ in the Lemma \ref{Sharpness for bi-disc}.
\end{proof}
The previous result is a symplectic phenomenon, the bound coming from the volume obstruction is much weaker. Using Proposition \ref{MainProp} and Corollary \ref{Capacity of a Lagrangian bi-disc}, we are able to show that the function $\rho_\mathcal{W}$ is squeezed in between $C_1 d_g$ and $C_2 d_g$, where $d_g$ is a Riemannian distance, and $C_i>0, i\in \{1,2\}$ are positive constants. 
\begin{proposition}\label{Equivalence of metrics}
There exist constants $C_i(\mathcal{W},g)>0$ such that for any $q_0, q_1 \in N$ we have $$C_1 d_g(q_0, q_1) \leq \rho_{\mathcal{W}}(q_0, q_1) \leq C_2 d_g(q_0, q_1).$$
\end{proposition}
We prove this proposition in Section \ref{ProofMain}. A construction of a symplectic embedding $e:B(2\sqrt{ab/\pi}) \to P_L^{2n}(a,b)$ appears in earlier work, \cite[Proposition 3.1]{Ka21}. The embedding from \cite{Ka21} is not explicit, and it doesn't necceseraly preserve real and imaginary parts. In Lemma \ref{Sharpness for bi-disc} we construct an explicit embedding that preserves the form $\lambda_{st} = \sum p_i dq_i$, and hence maps all vertical fibers of the ball to the vertical fibers of $P_L^{2n}(a,b)$, and preserves the real part, which is essential for the proof of Theorem \ref{Main}.  We will see in Proposition \ref{StrongViterboProp} that the existence of a symplectic embedding $e:B(2\sqrt{ab/\pi}) \to P_L^{2n}(a,b)$ implies the strong Viterbo conjecture about normalized capacities for $P^{2n}_L(a,b)$. Recall that a normalized symplectic capacity is a map $c: \mathcal{P}(R^{2n}) \to [0, +\infty]$ with the following properties:
\begin{itemize}
\item (\textit{Conformality}) $c(aX)=a^2c(X)$,
\item (\textit{Monotonicity}) If there is a symplectic embedding $\psi:X_1 \to X_2$ than  $c(X_1) \leq c(X_2)$,
\item (\textit{Normalization})  $c(B^{2n}(1)) = c(B^2(1) \times \mathbb{R}^{2n-2}) = \pi$.
\end{itemize}
 Examples of normalized capacities are the Gromov width: 
 $$Gr(X):=\sup \{ \pi r^2 \mid e:B^{2n}(r) \to X, e^* \omega_{st} = \omega_{st} \},$$ 
 and the \textit{cylindrical} capacity:
 $$c_Z(X):= \inf \{ \pi r^2 \mid e:X \to Z:= B^2(r) \times R^{2n-2}, e^* \omega_{st} = \omega_{st} \}.$$ 
 Construction of some other examples of normalized capacities appears in \cite{HZ90, EH89, Vi99, GH18}. The monotonicity and the normalization axiom imply that any other normalized capacity $c$ satisfies $Gr(X) \leq c(X) \leq c_Z(X),$ for any $X\subset \mathbb{R}^{2n}$. The strong Viterbo conjecture states that:
 \begin{conjecture}[Strong Viterbo conjecture]
 All normalized capacities of a convex domain $X\subset \mathbb{R}^{2n}$ are equal.
 \end{conjecture}
 In fact, this conjecture implies the conjecture by Viterbo from \cite{Vi00}: 
 \begin{conjecture}
 If $X\subset \mathbb{R}^{2n}$ is convex, and $c$ is normalized symplectic capacity then $c(X) \leq (n! Vol(X))^{1/n}.$
 \end{conjecture}
 Results related to this topic can be found in \cite{AAKO14, GHR22}. Our aim here is the next result which can be also deduced from \cite{Ka21}.
 
 \begin{proposition}\label{StrongViterboProp}
 The strong Viterbo conjecture holds for $P_L^{2n}(a,b)$.
 \end{proposition}
 
 \begin{proof}
 First, let us give a more elementary proof\footnote{This was pointed out to the author by Egor Shelukhin} that the Gromov width $Gr(P^{2n}_L(a,b))$ is bounded from above by $4ab$.  Take a projection $\pi_1:P^{2n}_L(a,b) \to \mathbb{R}^2$ given by $\pi_1(q,p) = (q_1, p_1)$. The image $\pi_1(P^{2n}_L(a,b))$ is contained in a square of area $4ab$, since the square is symplectomorphic to the disc, we get $c_Z(P^{2n}_L(a,b)) \leq 4ab$, and hence $Gr(P^{2n}_L(a,b))\leq 4ab$. The previous proof is a special case of \cite[Remark 4.2]{AAKO14}. Now, in Lemma \ref{Sharpness for bi-disc} we construct a symplectic embedding from a ball of capacity $4ab$ to $P^{2n}_L(a,b)$, hence
$$Gr(P^{2n}_L(a,b)) = 4ab.$$
We showed that $Gr(P^{2n}_L(a,b))=c_Z(P^{2n}_L(a,b))=4ab$ which ends the proof. 
 \end{proof}

\subsection{Structure of the paper}
In Section \ref{RelGrW}, we recall some previous results about the relative Gromov width, and we give the precise definition of the quantity $\widetilde{Gr}^k(N;D^*N)$. Section \ref{HW} contains an overview of Wrapped Floer homology, which is a Floer-type theory associated to a possibly non-compact Lagrangians $L$ in the completion of a Liouville domain $(M, -d\lambda)$, allowing Hamiltonians $H$ which are not compactly supported. We mainly follow \cite{Ab12}. In Section \ref{HM}, we cover the basics of the Morse theory of the space of paths $\mathcal{P}(q_0,q_1)$ on $N$ with fixed endpoints. We also recall the definitions of the isomorphisms
\begin{align*}
    &\Theta:HM(\mathcal{P}(q_0,q_1)) \to HW(T^*_{q_0} N, T^*_{q_1} N), \\ 
    &\mathcal{F}:HW(T^*_{q_0} N, T^*_{q_1} N) \to HM(\mathcal{P}(q_0,q_1)).
\end{align*}
The isomorphism $\Theta$ is constructed in \cite{APS08, ASch10}, and the isomorphism $\mathcal{F}$ is constructed in \cite{Ab12}. It was also shown in \cite{Ab12} that $\Theta \circ \mathcal{F} = Id$. In Section \ref{Technical results}, we prove the existence of a perturbed $J$-holomorphic half-strip $u_\epsilon:(-\infty, 0] \times [0,1] \to T^*N$ with boundary conditions on the fibers $T^*_{q_0} N$ and $T^*_{q_1} N$, and the  zero section $N$. The curve $u_\epsilon$ is $J$-holomorphic on $D^*N$, which is essential for our application. We prove that the energy of $u_\epsilon$ satisfies $E(u_\epsilon)=d_g(q_0,q_1)+\epsilon$, for any $\epsilon>0$ small enough. This follows from the properties of $\mathcal{F}$ and $\Theta$ in the case when $q_0$ and $q_1$ are connected by a unique length minimizing geodesic. If $q_0$ and $q_1$ belong to the sufficiently small geodesically convex neighborhood in $N$ a similar existence result was earlier obtained using a different technique by Abouzaid in \cite[Chapter 13.5]{Ab15}. This result doesn't seem applicable in our situation, since the curves $u$ from \cite{Ab15} are not $J$-holomorphic on $D^*N$. In Section \ref{MainPropProof}, we prove Proposition \ref{MainProp}. There, we recall the fact from Riemannian geometry that for a generic $q_1$ there is a unique length minimizing geodesic from $q_0$ to $q_1$. We use this property to show that, without loss of generality, we can assume that the centers $q_i=e_i(0)$ of the balls $e_i:B^{2n}(r_i) \to D^*N$ are joined by a unique length minimizing geodesic. By a careful choice of an almost complex structure $J$, and using the monotonicity theorem for minimal surfaces, we bound $\pi r_1^2 + \pi r_2^2$ in terms of $E(u)$. In Section \ref{ProofMain}, we give a precise definition of the function $\rho_\mathcal{W}$, and we construct an explicit symplectic embedding 
$$e:B^{2n}\left(2\sqrt{\frac{ab}{\pi}}\right) \to P_L(a,b),$$
from the ball of capacity $4ab$ to the Lagrangian bi-disc. We use this embedding to prove Proposition \ref{Equivalence of metrics}. In the case $\mathcal{W}=D^*M$, using such an embedding,  we provide a good lower bound for $\rho_\mathcal{W}$ when $q_0$ and $q_1$ are close enough. We conclude the section with the proof of Theorem \ref{Main}.

\textbf{Acknowledgements.} I want to express my gratitude to my advisor Octav Cornea, who introduced me to the problem of relative symplectic embeddings and carefully followed the development of this paper, always asking the right questions and guiding me in the right direction. I am also very thankful to my co-advisor, Egor Shelukhin, for many fruitful discussions. I am grateful to Dylan Cant for helping me draw the Figures with TikZ. Finally, I thank Francois Charette, Yaron Ostrover, and Vinicius G. B. Ramos for their useful comments.  This work is a part of my Ph.D. thesis at the University of Montréal, where I am
partially supported by an ISM scholarship. This research is also partially supported by the Fondation
Courtois.

\section{Preliminaries}\label{Preliminaries}

\subsection{Relative Gromov width}\label{RelGrW}
 Fix a Lagrangian submanifold $L$ of a symplectic manifold $(M,\omega)$. The  \textit{relative Gromov width} is given by the following equation
$$
Gr(L;M)= \sup \{ \pi r^2 \mid \exists e:B^{2n}(r)\rightarrow M \text{ relative to } L \text{, and } e^* \omega = \omega_{st} \},
$$
where $\omega_{st}= \sum dq_i \wedge dp_i$ is the standard symplectic form on $\mathbb{R}^{2n}$. This notion was introduced in \cite{BaC06}. The \textit{Relative packing number} of $k$ balls is 
$$
Gr^k(L, M) = \sup \left\{ k\pi r^2  \Big\vert \exists e:\bigsqcup_{i=1}^kB^{2n}(r)\rightarrow M \text{ relative to } L \text{, } e^* \omega = \omega_{st} \right\}.
$$
A weaker form of the conjecture by Barraud and Cornea from \cite{BaC06} is that any displaceable closed Lagrangian has finite relative  Gromov width. Biran and Cornea gave bounds for the relative packing numbers in the monotone case and proved that displaceable monotone Lagrangians have finite relative Gromov width in \cite{BC09}. The conjecture in the form from \cite{BaC06} was proven by Charette in \cite{Ch12} for monotone Lagrangians. It also holds for orientable two-dimensional Lagrangians  (\cite{Ch15}). In \cite{BM14}, Borman and McLean proved that closed, displaceable, orientable Lagrangians, which admit metrics with non-positive sectional curvature, have finite relative Gromov width. All these results use Floer-theoretic machinery to produce J-holomorphic curves to bound the radius. Surprisingly, Dimitroglou-Rizzel proved in \cite{Ri15} that Lagrangian embeddings in $\mathbb{C}^n$ constructed by Ekholm, Eliashberg, Murphy, and Smith in \cite{EEMS13} have infinite relative Gromov width. These examples are flexible in the sense that their construction uses a certain h-principle. 

From \cite[Proposition  1.8.]{BM14} it follows that
$$Gr(L;M,\omega) = \sup \{Gr(L; \mathcal{W}_L, \omega) \mid \mathcal{W}_L \text{ is a Weinstein neighborhood} \}.$$  
This equality implies that the relative Gromov width is a ``natural" invariant to measure how big a Weinstein neighborhood can be. On the other hand, it suggests that it is worthwhile considering symplectic embeddings to open neighborhoods of the zero section $N$ inside $T^*N$. Now, we define a variant of the relative Gromov width inside $T^*N$, which resembles the invariant $Gr(L_1,L_2;M)$ introduced in \cite{Le07}.

\begin{definition}
The relative packing number of $k$ balls in the unit-disc bundle $D^*N$ is
\begin{equation*}
    \widetilde{Gr}^k(N; D^*N) := \sup \left\{ k\pi r^2  \left. \begin{aligned}
        \text{}\\ \text{}\\ \text{}
    \end{aligned} \right\vert 
        \begin{aligned}
           &\exists e: \bigsqcup_{i=1}^k B_i^{2n}(r) \to D^*N,  \quad e^* \omega = \omega_{st}, \\
           &e^{-1}(L) = \bigsqcup B_i^{n}(r)  \times \{0\}, \quad q_i:= e(0_i)\\
           & \quad e^{-1}\left(\bigsqcup D^*_{q_i}N\right) =\bigsqcup \{0\} \times B_i^{n}(r).  
        \end{aligned} \right\}.
\end{equation*}
In particular, for $k=1$ we call it the Relative Gromov width of the unit-disc bundle $\widetilde{Gr}(N;D^*N)$.
\end{definition}

Note that $\widetilde{Gr}(N;D^*N)= \sup \{ Gr(N, D^*_q N; D^*N) \mid q \in N\}$.
\begin{remark}
Since $Ham_c(D^*N)$ acts transitively on the zero section $N$, Proposition \ref{MainProp} doesn't hold if the embedding $e$ doesn't satisfy the constraint \ref{ImPartCond} on the imaginary parts. The centers $q_i$ of the balls $e(B^{2n}(r_i) )$ can be arbitrarily close in the absence of this constraint. Still, it makes sense to ask whether the assumption about the imaginary part can be removed in Corollaries \ref{RelGr} and \ref{RelGr2}.  The inequalities $\widetilde{Gr}(N;D^*N) \leq Gr(N; D^*N) \leq Gr(D^*N)$ are straightforward from the definitions of the respective versions of the Gromov width. In the existing examples, to the author`s knowledge, the Relative Gromov width $\widetilde{Gr}(N; D^*N)=Gr(D^*N)$ is equal to the full Gromov width, without constraints on the real and the imaginary part. It would be interesting to construct an example where any of the inequalities from above are strict or to prove that these are in fact equalities. It is easy to construct examples with $Gr(N; \mathcal{W})<Gr(\mathcal{W})$, where $\mathcal{W}$ is some non-symmetric open neighborhood of the zero section $N$.
\end{remark}
The following examples are from \cite[Section 6]{KS21}. It is easy to see that these relative embeddings can be chosen to satisfy the condition (\ref{ImPartCond}).
\begin{example}

   (i) Let $(S^n, g)$ the standard sphere $S^n \subset \mathbb{R}^n$ scaled so that $diam (S^n) =1/2$. There are two disjoint relative embeddings $e_i:B^{2n}(1/\sqrt{\pi} ) \to D^*S^n$, hence we get that $\widetilde{Gr}^2(S^n; D^*S^n)=2=4\cdot diam(S^n).$ These two embeddings $e_i$ actually fill out the volume of $D^* S^n$
    
(ii) Let $(\mathbb{R}{P}^n, g)$ be the real projective space with the metric induced from the standard sphere $S^n$, scaled so that $diam (\mathbb{R}{P}^n) =1/4$. There exist a relative symplectic embedding $e:B^{2n}(1/\sqrt{\pi} ) \to D^*\mathbb{R}{P}^n$, which fills out the volume of $D^*\mathbb{R}{P}^n$. We get $\widetilde{Gr}(\mathbb{R}{P}^n;D^*\mathbb{R}{P}^n)= 1 = 4\cdot diam (\mathbb{R}{P}^n).$
\end{example}
In these examples, $diam (N)$ is equal to the injectivity radius $\rho_{inj}$ of $(N,g)$. In the proof of the Proposition \ref{Equivalence of metrics}, we see that if $\rho_{inj} < diam (N)$ the lower bound for $\rho_\mathcal{W} (q_0,q_1)$ is strictly smaller than $d_g(q_0,q_1)$ . Even though the lower bound in the Proposition \ref{Equivalence of metrics} is not necessarily sharp, it suggests that when $\rho_{inj}$ is much smaller than $diam(N)$, inequalities from the Corollaries \ref{RelGr} and \ref{RelGr2} should not be sharp. Indeed, we can always find a metric $g$ on $N$ with $Vol_g(N)=1$ and $diam_g(N)$ arbitrarily big (hence $\rho_{inj}$ is arbitrarily small). In that case, trivial volume obstruction gives a better upper bound on $\widetilde{Gr}^k(N;D^*N)$ than Corollaries \ref{RelGr} and \ref{RelGr2}.

\subsection{Wrapped Floer homology}\label{HW}

In this section, we recall the definition of the Wrapped Floer homology. We mainly follow \cite{Ab12}, except we have changed some sign notation. Hence, we provide proofs of some standard results for the sake of completeness.  Let $(\Bar{M}, -d \Bar{\lambda}, X_{\Bar{\lambda}})$ be a compact Liouville domain i.e. $\Bar{\omega} = -d\Bar{\lambda} $ is a symplectic 2-form, $\alpha:=\Bar{\lambda} \vert_{\partial \Bar{M}}$ is a contact form on $\partial \Bar{M}$ and Liouville vector field $X_{\Bar{\lambda}}$, given by $i_{X_{\Bar{\lambda}}} \Bar{\omega} = -\Bar{\lambda}$, 
is positively transverse to the contact boundary $(\partial \Bar{M}, \alpha)$. Liouville domain admits a positive completion defined by $M:= \Bar{M} \sqcup_{\partial \Bar{M}} \partial \Bar{M} \times [1,+\infty),$
with a symplectic form $\omega = - d\lambda$, where $\lambda = r \alpha$ for $r\in [1,+\infty)$.  For a positive constant $C>0$, we define a class of \textit{admissible} Hamiltonians to be the set $\mathcal{H}_C = \{ H \in C^{\infty}(M) \mid H(x,r) = C r^2, \text{ for } r\geq 2 \}.$ We will use almost complex structures $J\in End(TM, TM)$, which are compatible with $\omega$ and satisfy $\lambda \circ J = - dr,$ for $r\geq 2$. Let $\mathcal{J}$ be the space of all such almost complex structures. For $J \in \mathcal{J}$, we have a maximum principle for perturbed $J$-holomorphic curves, which is needed to achieve the compactness of relevant moduli spaces.

Lagrangian submanifold $L \subset M$ is \textit{exact} if there is a primitive $f:L \to \mathbb{R}$ of the Liouville form $\lambda \vert_L$ restricted to $L$ i.e. $df = \lambda \vert_L.$ We say that $L$ is \textit{cylindrical} if there is a non-empty Legendrian submanifold $\Lambda \subset \partial \Bar{M}$ such that $L \cap \partial \Bar{M} \times [1,+\infty) = \Lambda \times [1, +\infty).$
Since $\alpha\vert_{\Lambda}=0$, a primitive $f$ for L is locally constant on $r\geq1.$

Fix two exact Lagrangians $(L_0, f_0)$ and  $(L_1, f_1)$  which are either closed or cylindrical. Fix $H \in \mathcal{H}_C$, and let $X_H$ be its Hamiltonian vector field, defined by the equation $i_{X_H} \omega = dH.$
For $x:[0,1] \to M$, $x(0)\in L_0$ and $x(1)\in L_1$ we define an action functional
$$
\mathcal{A}_H(x) = \int x^* \lambda - \int_0^1 H(x(t)) dt - f_1(x(1)) + f_0(x(0)).
$$
Critical points of $\mathcal{A}_H$ are solutions of the Hamiltonian equation $x'=X_H \circ x,$ such that $x(0) \in L_0$, $x(1) \in L_1$. Indeed, for a variation $\xi$ of a path $x:[0,1] \to M$ from $L_0$ to $L_1$ we have
$$
d \mathcal{A}_H(\xi) = \int_0^1 \omega(x'(t) - X_H(x(t)),\xi(t)) dt.
$$
\begin{remark}\label{ActionH=f}
If $H(x,r)=f(r)$ for $r\geq 1$ and if $y \in Crit(\mathcal{A}_H)$ is such that  $y(t) \notin  \Bar{M}$ for all $t\in [0,1]$ then
$$
\mathcal{A}_H(y) = r f'(r) - f(r).
$$
\end{remark}
We will assume that all $x \in Crit(\mathcal{A}_H)$ are non-degenerate which means that if $\phi^1_H$ is a time one map of the Hamiltonian vector field $X_H$ then $\phi^1_H(L_0) \pitchfork L_1.$ We define the Wrapped Floer complex of $L_0$ and $L_1$ as
$$
CW(L_0, L_1; H, J) = \bigoplus_{x \in Crit(\mathcal{A}_H)} \mathbb{Z}_2 \langle x \rangle.
$$
For $x_-, x_+ \in Crit(\mathcal{A}_H)$ we define a moduli space $\widetilde{\mathcal{M}}(x_-,x_+;H,J)$ to be the set  of maps $u:\mathbb{R}\times [0,1] \to M$ which satisfy a perturbed $J$-holomorphic equation
\begin{equation}\label{FloerEQ}
    \frac{\partial u}{\partial s}+J\left(\frac{\partial u}{\partial t}  - X_H\right)=0,
\end{equation}
with the following boundary and asymptotic conditions for $i \in \{0,1\}$:
\begin{equation}\label{Bdry}
    \begin{cases}
    &u(\mathbb{R} \times \{i\}) \subset L_i, \\
    &\lim\limits_{s\to \pm \infty} u(s,t) = x_{\pm}(t).
    \end{cases}
\end{equation}
For $u\in \widetilde{\mathcal{M}}(x_-,x_+;H,J)$ we define its energy by the equation
$$
E(u)= \iint_{\mathbb{R} \times [0,1]} \omega \left(\frac{\partial u}{\partial s}, J \frac{\partial u}{\partial s}\right)dsdt.
$$
\begin{lemma}\label{Energy}
For $u \in \widetilde{\mathcal{M}}(x_-, x_+;H,J)$  we have $E(u)=\mathcal{A}_H(x_-) - \mathcal{A}_H(x_+).$
\end{lemma}
\begin{proof}
Set $Z:=\mathbb{R}\times [0,1]$. Since $u$ satisfies the equation (\ref{FloerEQ}) we have
\begin{align*}
    E(u)&= \iint_{Z} \omega\left(\frac{\partial u}{\partial s}, J \frac{\partial u}{\partial s} \right) ds dt = \iint_{Z} \omega\left(\frac{\partial u}{\partial s},\frac{\partial u}{\partial t}  - X_H \right) ds dt \\
&= \int u^* \omega + \iint_{Z} dH\left(\frac{\partial u}{\partial s} \right) ds dt = -\int u^*  d\lambda + \iint_{Z} \frac{\partial}{\partial s}(H(u(s,t)) ds dt.
\end{align*}
Because of the boundary conditions (\ref{Bdry})  and since $\lambda\vert_{L_i}=df_i$, after applying Stokes' theorem we get that
\begin{align*}
    E(u) &= \int x_- ^* \lambda -f_1(x_-(1)) + f_0(x_-(0)) - \int_0^1 H(x_-(t)) dt\\
    &  - \int x_+^* \lambda+f_1(x_+(1)) -f_0(x_+(0)) + \int_0^1 H(x_+(t))dt\\
    &= \mathcal{A}_H(x_-) - \mathcal{A}_H(x_+).
\end{align*}
\end{proof}
Now we will show that elements $\widetilde{\mathcal{M}}(x_-,x_+)$ cannot escape to infinity.
\begin{lemma}\cite{Ab12}\label{Escape}
For $x_\pm \in Crit(\mathcal{A}_H)$ there is a compact set $K \subset M$ such that all elements $u \in \mathcal{M}(x_-, x_+)$ satisfy $Im(u) \subset K.$
\end{lemma}
\begin{proof}
Take $r\geq 2$ which separates $x_-$ and $x_+$ from infinity, and assume that $Im (u) \cap \partial \Bar{M}\times (r, +\infty) \neq \emptyset$. Hence, there is a compact surface $S:= u^{-1}(\partial \Bar{M}\times [r, +\infty)) \subset \mathbb{R} \times [0,1]$ whose boundary has two parts:
$$
\partial_r S = u^{-1}(\partial \Bar{M}\times \{r\}) \text{ and } \partial_l S = S \cap u^{-1}(L_0 \cup L_1).
$$
Now we have
\begin{align*}
    0 < E(u\vert_S) &= \iint_{S} \omega\left(\frac{\partial u}{\partial s}, J \frac{\partial u}{\partial s} \right) ds dt = -\iint_S u^* d\lambda + \iint_S u^*(dH\wedge dt) \\
    &=-\int_{\partial S} u^* \lambda + \int_{\partial S} u^*H dt = -\int_{\partial_r S} u^*\lambda + \int_{\partial_r S} u^*H dt.
\end{align*}
Last equality holds since $\lambda\vert_{L_i} = 0$ and since $dt = 0$ on $\partial_l S$. Since $du-X_H \otimes dt = -J\circ(du-X_H \otimes dt) \circ i$ we have that 
$$
- \lambda \circ du \vert_{\partial_r S} = \lambda \circ J \circ (du-X_H \otimes dt) + \lambda(X_H) dt = \lambda \circ J \circ du \circ i + 2Cr^2 dt.
$$
The last equality holds because $\lambda(J X_H) = 0$ and $\lambda (X_H) = 2C r^2$. This implies
\begin{align*}
    0<E(u)= \int_{\partial_r S} (\lambda \circ J) \circ du \circ i + \int_{\partial_r S} 3C r^2 dt= - \int_{\partial_r S} dr \circ du \circ i,
\end{align*}
because $\int_{\partial_r S}dt=0$ and $\lambda \circ J = - dr$ . On the other hand, if $\xi \in T \partial_r S$ is positively oriented then $i \xi$ points inwards $S$ and hence $du(i \xi)$ points outwards on the boundary $\partial \Bar{M}$, hence $dr (du(i\xi)) \geq 0,$
which leads to the contradiction.\end{proof}

The moduli space $\widetilde{\mathcal{M}}(x_-,x_+;H,J)$ admits an $\mathbb{R}$-action by translation in the $s$ coordinate. Let $\mathcal{M}(x_-,x_+)$ be the quotient of $\widetilde{\mathcal{M}}(x_-,x_+;H,J)$ by the $\mathbb{R}$ action. Because of the Lemmas \ref{Energy} and \ref{Escape}, we can apply Gromov's compactness theorem (\cite{Gr85}) for moduli space $\mathcal{M}(x_-,x_+)$. The exactness of the symplectic manifold $M$ and the exactness of the Lagrangian submanifolds $L_0$ and $L_1$ exclude the possibility of bubbling phenomena. The regularity of the moduli space $\mathcal{M}(x_-,x_+)$ follows from standard results (\cite{Fl88, FHS94}). It is achieved by perturbing an admissible almost complex structure $J \in \mathcal{J}$. We define the differential $d: CW(L_0, L_1) \to CW(L_0,L_1)$ on the generators to be $$d x_- = \sum_{\substack{x_+ \\ dim \mathcal{M}(x_-,x_+) = 0}} \#_2 \mathcal{M}(x_-, x_+) x_+,$$ and extend it to $CW(L_0,L_1)$ by linearity. By standard results (\cite{Fl88, FHS94}) it follows that $d^2 = 0.$

Homology of the complex $CW(L_0, L_1; H, J)$ does not depend on the choice of $H \in \mathcal{H}_C$. To see that, fix $H_1, H_2 \in \mathcal{H}_C$ and take a smooth path $H_s \in \mathcal{H}_C$ of Hamiltonians for $s \in \mathbb{R}$, such that $H_s (x) = 
H_1(x)$ for  $s\leq 0$, and $H_2(x)$ for  $s\geq 1$. For $x \in Crit(\mathcal{A}_{H_1})$ and $y \in Crit(\mathcal{A}_{H_2})$ define a moduli space $\mathcal{M}(x,y;H_s,J)$ to be the set of maps $u:\mathbb{R}\times [0,1] \to M$ such that 
\begin{equation*}
    \begin{cases}
    & \frac{\partial u }{ \partial s} + J \left(\frac{\partial u }{ \partial s} - X_{H_s} \right) = 0, \\
    &u(\mathbb{R} \times \{i\}) \subset L_i, \\
    &\lim\limits_{s\to - \infty} u(s,t) = x(t),\\
    &\lim\limits_{s\to + \infty} u(s,t) = y(t).
    \end{cases}
\end{equation*}
A count of  rigid objects from the zero-dimensional moduli space $\mathcal{M}(x,y)$ defines a continuation map
$$
\Phi_{H_s}: CW(L_0, L_1; H_1, J) \to CW(L_0, L_1; H_2, J).
$$
It is a standard result that the map $\Phi_{H_s}$ is a chain map and induces an isomorphism in homology. By a similar argument, $HW(L_0, L_1; H,J)$ does not depend on an almost complex structure $J \in \mathcal{J}$.
\begin{remark}
Since we work in the ungraded setting and we work with $\mathbb{Z}_2$ coefficients, we do not need to impose any other assumptions on $M$, $L_0$, and $L_1$.
\end{remark}
For the proof of the Proposition \ref{MainProp}, our ambient symplectic manifold will be a cotangent bundle $T^*N$ of a closed Riemannian manifold $(N,g)$. We will consider Wrapped Floer homology $HW(L_0,L_1)$  for the two fibers $L_0:=T_{q_0}^*N$ and $L_1:=T_{q_1}^*N$. Any fiber $T_q^*N$ is an exact, cylindrical Lagrangian and the canonical Liouville form $\lambda$ vanishes on the fiber, hence the primitive $f$ can be chosen to be identically equal to $0$. 
The Hamiltonian function given by $H(p) = C \|p\|^2_g$
belongs to $\mathcal{H}_C$ and for a generic choice  of $q_1\in N$  we have
$\phi_H^1(T_{q_0}^*N) \pitchfork T_{q_1}^*N.$ For such a Hamiltonian $H$, Hamiltonian chords of $X_H$ are lifts of the geodesics from $q_0$ to $q_1$ since the flow $\phi_H^t$ of $X_H$ is just a reparametrization of the co-geodesic flow.

\subsection{Morse theory for the space of paths}\label{HM}
Let $(N,g)$ be a closed Riemannian manifold, and let $\mathcal{P}(q_0,q_1)$ be a $W^{1,2}$ completion of the space of smooth paths $x:[0,1] \to N$ such that $x(0)=q_0$ and $x(1)=q_1$. Lagrangian function $L:TN \to \mathbb{R}$ induces a functional $\mathcal{L}:\mathcal{P}(q_0,q_1) \to \mathbb{R}$
on the space of paths defined by
$$
\mathcal{L}(x)=\int_0^1 L(x(t),x'(t)) dt.
$$
\textit{Fenchel} dual of $L$ is $H:T^*N \to \mathbb{R}$ defined by
$$
H(q,p)= \max_{v \in T_q N} p(v) - L(q,v).
$$
If we set $L(v):= \frac{1}{4C} \|v\|_g^2$, then it's Fenchel dual is $H(q,p)=C \|p\|^2_g.$ Such $L$ and $H$ satisfy respectively the conditions (L1), (L2), and (H1),(H2), from Sections 2.1 and 3.1 in \cite{ASch10}. Following \cite{APS08, ASch10}, we define the Morse complex for $\mathcal{P}(q_0,q_1)$ as
\begin{equation*}
    CM(\mathcal{P}(q_0,q_1)) = \bigoplus_{\gamma \in Crit(\mathcal{L})} \mathbb{Z}_2 \langle \gamma \rangle.
\end{equation*}
Let $X$ be a smooth pseudo-gradient for $\mathcal{L}$ and for critical point $\gamma \in Crit(\mathcal{L})$, such a vector field exists by \cite{ASch09}. Define stable and unstable manifolds of $\gamma$ to be
\begin{align*}
    W^s(\gamma, X) &= \{ p \in \mathcal{P}(q_0, q_1) \mid \lim_{t \to +\infty } \phi_t(p) = \gamma  \} \\
    W^u(\gamma, X) &= \{ p \in \mathcal{P}(q_0, q_1) \mid \lim_{t \to -\infty } \phi_t(p) = \gamma \}
\end{align*}
For a generic choice of the Riemannian metric $g$ we have that $W^u(\gamma_-) \pitchfork W^s(\gamma_+)$ for all $\gamma_-, \gamma_+ \in Crit(\mathcal{L})$. This intersection is a finite-dimensional manifold whose dimension is equal to the difference $i(\gamma_-) - i(\gamma_+),$ where $i(\gamma)$ is the Morse index of a geodesic $\gamma$. Since we work in the ungraded setting, we do not keep track of the indices. The manifold $W^u(\gamma_-) \cap W^s(\gamma_+)$ admits a free action by $\mathbb{R}$ and we define $\mathcal{M}(\gamma_-, \gamma_+)$ to be the quotient by this action. The differential $\partial: CM(\mathcal{P}(q_0,q_1)) \to CM(\mathcal{P}(q_0,q_1)),$ is defined by 
$$
\partial \gamma_- = \sum_{\substack{\gamma_+ \\ dim \mathcal{M}(\gamma_-,\gamma_+) = 0}} \#_2 \mathcal{M}(\gamma_-, \gamma_+) \gamma_+.
$$
From the standard techniques, it follows that $\partial \circ \partial = 0$.  The Morse complex $CM(\mathcal{P}(q_0,q_1))$ is quasi-isomorphic to the Wrapped Floer complex of two fibers $CW(T_{q_0}^*N, T_{q_1}^*N),$ using counts of geometric objects (see \cite{APS08}). In \cite[Section 5]{Ab12}, it was shown that this quasi-isomorphism admits a right inverse, which is also defined by a suitable count of geometric objects. Now we formulate the theorem which combines these results.
\begin{theorem}\cite{APS08, Ab12}\label{Isomorphism}
For a closed Riemannian $(N,g)$ and $q_0, q_1 \in N$ there are quasi-isomorphisms
\begin{align*}
    &\Theta: CM(\mathcal{P}(q_0, q_1) ) \to CW(T_{q_0}^*N, T_{q_1}^*N),\\
    &\mathcal{F}: CW(T_{q_0}^*N, T_{q_1}^*N) \to CM(\mathcal{P}(q_0, q_1)),
\end{align*}
such that $\Theta \circ \mathcal{F}$ is chain homotopic to the identity.
\end{theorem}
We will spend the rest of this section defining $\Theta$ and $\mathcal{F}$. For $\gamma \in Crit(\mathcal{L})$ and $x \in Crit(\mathcal{A}_H)$ we define $\mathcal{M}^{\Theta}(\gamma, x)$ to be the set of maps $u:[0,+\infty) \times [0,1] \to T^*N$ such that

\begin{equation*} 
    \begin{cases}
    &\frac{\partial u}{\partial s}+J\left(\frac{\partial u}{\partial t}  - X_H\right)=0, \\
    &u([0,+\infty) \times \{i\}) \subset T_{q_i}^*N, \\
    &\lim\limits_{s\to +\infty} u(s,t) = x(t),\\
    &\pi\circ u(0,t) \in W^u(\gamma).
    \end{cases}
\end{equation*}

Map $\Theta: CM(\mathcal{P}(q_0, q_1) ) \to CW(T_{q_0}^*N, T_{q_1}^*N)$ is defined by

\begin{equation*}
    \Theta(\gamma)= \sum_{\substack{x \in Crit(\mathcal{A}_H) \\ dim \mathcal{M}^{\Theta}(\gamma,x) = 0}} \#_2 \mathcal{M}^{\Theta}(\gamma,x) x.
\end{equation*}
It follows from \cite{APS08, ASch10} that $\Theta$ is a well-defined chain map. For $u \in \mathcal{M}^{\Theta}(\gamma,x)$ we define its energy by the equation
$$
E(u):= \int_0^1 \int_0^\infty \omega\left(\frac{\partial u}{\partial s}, J \frac{\partial u}{\partial s} \right) ds dt.
$$

\begin{lemma}\cite{ASch05}  \label{A<L}
If $H$ is the Fenchel dual of $L$ then  $\mathcal{A}_H(x) \leq \mathcal{L}(\pi \circ x)$ for $x\in C^{\infty}([0,1]; T^*N)$. Moreover, the equality holds if $x \in Crit(\mathcal{A}_H)$. 
\end{lemma}
\begin{proof}
\begin{align*}
\mathcal{A}_H(x) &=\int_0^1 \lambda(x'(t)) - H(x(t)) dt = \int_0^1 \lambda(x') - \max_{v}\left( x(v) - L(\pi\circ x ,v) \right) dt \\ 
&= \int_0^1 \min_{v}\left(x ( \pi_* x') - x(v) + L(\pi \circ x,v) \right) dt\\
&\leq \int_0^1 L(\pi \circ x, \pi_* x') dt = \mathcal{L}(\pi \circ x) .
\end{align*}
The last inequality follows from setting $v= \pi_* x' = (\pi \circ x)'$. It follows from the definition of $H$ that $H(q,p)=p( D_{vert}^{-1}L (p) )- L(q, D_{vert}L^{-1}(p) )$ where $D_{vert}L:TN \to T^*N$ is the derivative of $L$ in the fiber direction, called the Legendre transformation. So, equality holds iff $x(t)=(q(t), p(t))$ is such that $D_{vert} L ( q'(t) ) = p(t)$. In particular, since the Legendre transformation gives a one-to-one correspondence between the critical points of $\mathcal{L}$ and solutions of the Hamiltonian equation, equality holds if $x'=X_H \circ x$.
\end{proof}
From Lemmas \ref{Energy} and \ref{A<L} we have that the energy of $u\in \mathcal{M}^{\Theta}(\gamma, x)$ satisfies $E(u)\leq \mathcal{L}(\gamma) - \mathcal{A}_H(x).$ This inequality is essential for the proof of compactness results for the moduli space $\mathcal{M}^{\Theta}(\gamma,x)$, and it also guarantees that $\Theta$ is quasi-isomorphism by diagonal argument (see \cite[Section 3.5]{ASch05}).

Define $\mathcal{H}^{rel}_C \subset \mathcal{H}_C$ to be the set of $H$ which are constant on $N \subset T^*N$.
Fix $H\in \mathcal{H}^{rel}_C$,  for $x \in Crit(\mathcal{A}_H)$, following \cite{Ab12} we define the moduli space $\mathcal{M}^{\mathcal{F}_H}(x, \gamma)$ to be the set of maps $u:(-\infty,0]\times[0,1] \to T^*N$ which are satisfying the following conditions 
\begin{equation*}
    \begin{cases}
    &\frac{\partial u}{\partial s}+J\left(\frac{\partial u}{\partial t}  - X_H\right)=0, \\
    &u((-\infty,0] \times \{i\}) \subset T_{q_i}^*N, \\
    &u_0(t):=u(0,t) \in N \subset T^*N,\\
    &\lim\limits_{s\to -\infty} u(s,t) = x(t), \quad u_0 \in  W^s(\gamma).
    \end{cases}
\end{equation*}
Now, we define a map $\mathcal{F}_H: CW(T_{q_0}^*N, T_{q_1}^*N;H,J) \to CM(\mathcal{P}(q_0, q_1), \mathcal{L})$ by the equation

\begin{equation*}
    \mathcal{F}_H(x)= \sum_{\substack{\gamma \in Crit(\mathcal{L}) \\ dim \mathcal{M}^{\mathcal{F}_H}(x, \gamma) = 0}} \#_2 \mathcal{M}^{\mathcal{F}_H}(x,\gamma) \gamma.
\end{equation*}
As in Lemma \ref{Energy} we have that every $u\in \mathcal{M}^{\mathcal{F}_H}(x, \gamma)$ satisfies
\begin{equation} \label{EnergyF}
   E(u)=\mathcal{A}_H(x) - \mathcal{A}_H(u_0)= \mathcal{A}_H(x) + H\vert_N. 
\end{equation}
Last equality holds since $H$ is constant on $N$ and $\lambda\vert_N = 0$.

\section{Proofs}

\subsection{Existence of a $J$-holomorphic curve}\label{Technical results}

Assume that there is a unique length minimizing geodesic $\Bar{\gamma}:[0,1] \to N$ from $q_0$ to $q_1$. Define a Hamiltonian $H:T^*N \to \mathbb{R}$ by setting
$$
H(p)=\frac{d_g (q_0,q_1)}{2} \|p\|^2_g.
$$
Let $\delta>0$ be such that if $\gamma \in \mathcal{P}(q_0,q_1)$ is a geodesic, and $\gamma \neq \Bar{\gamma}$ then $l(\gamma)>d_g (q_0,q_1)(1+\delta).$ Let $\Bar{x}$ be the chord which corresponds to the geodesic $\Bar{\gamma}$. Using Remark \ref{ActionH=f}, action of $\Bar{x}$ satisfies $\mathcal{A}_{H}(\Bar{x}) = d_g (q_0,q_1) / 2$, and for any chord $y\neq \Bar{x}$, we have $\mathcal{A}_H(y) \geq  d_g(q_0, q_1) (1+\delta)^2 /2$. 

Set $d:=d_g (q_0,q_1)$, for $\epsilon>0$ small enough let $f_{\epsilon}:[0,+\infty) \to [0,+\infty)$ be a smooth function, as in Figure (\ref{Functionf}), which satisfies the following conditions 
\begin{itemize}
    \item[1.)] $f_\epsilon(r)=\begin{cases}
    d/2+\epsilon, &r\leq 1\\
    d r^2/2, &r\geq1+\delta,
    \end{cases}$
\item[2.)] $f_{\epsilon}(r) \geq d r^2/2 $, and $f''_\epsilon (r)>0$ for $r\in (1, 1+\delta)$,
\item[3.)] if $r_{\epsilon}:=1+\epsilon / d $, $f_{\epsilon}'(r_{\epsilon})=d$ and $f_{\epsilon}(r_{\epsilon}) = d/2 + 3\epsilon /2$.
\end{itemize}
Such a function exists if $r_{\epsilon} < 1 + \delta$ and $d r_{\epsilon}^2 /2 \leq f_{\epsilon}(r_{\epsilon})  < d (1+\delta)^2 /2 $, i.e. if $\epsilon < \min\{d, \delta d/2\}$. Function $f_\epsilon$ is chosen in a way that there is a one-to-one correspondence between Hamiltonian chords of  $H(p)=d \|p\|^2_g /2$ and $H_{\epsilon}(p) = f_\epsilon (\| p\|_g) $. Define $x_\epsilon$ to be the unique Hamiltonian chord of $H_\epsilon$ which corresponds to $\Bar{\gamma}$. From the conditions $1.)$ and $2.)$ one can estimate that the action of $x_{\epsilon}$ satisfies $\mathcal{A}_{H_{\epsilon}}(x_{\epsilon}) \in (d/2 - \epsilon, d/2)$. The last condition is equivalent to $$\mathcal{A}_{H_{\epsilon}}(x_{\epsilon}) =  r_{\epsilon} f_{\epsilon}'(r_{\epsilon}) - f_{\epsilon}(r_{\epsilon})= \frac{d - \epsilon}{2}.$$
 \begin{figure}[ht]
    \centering
    \begin{tikzpicture}[xscale=2.5,yscale=2]
    \draw[line width=1pt,blue!80!pink] (-1.5,{0*0.4})--(0,{0*0.4}) to[in={201.8014},out=0] coordinate[pos=0.5](Y) (0.5,{0.4*0.25}) plot[domain=0.5:1.5,variable=\x] ({\x},{0.4*\x * \x});
    \foreach \x in {0,1} {
      \draw[dashed,black!50!white] ({\x/2},-0.25)coordinate(X\x)--({\x/2},{2.25*0.4});
    }
    \foreach \x in {15} {
      \path (Y)--+(\x:-2)coordinate(Z1);
      \path (Y)--+(\x:2) coordinate (Z2);
    }
    \begin{scope}
      \clip (-1.5,-0.25) rectangle (1.5,{2.25*0.4});
      \path[name path={P1}] (-1.5,-0.25)--+(2,0) (1.5,-0.25)--+(0,2);
      \path[name path={P2}] (Z1)--(Z2);
      \path[name intersections={of=P1 and P2}];
      \coordinate (W1) at (intersection-1);
      \coordinate (W2) at (intersection-2);
      \draw[red,line width=1pt] (Z1)--(Z2);
    \end{scope}

    \draw (-1.5,-0.25) rectangle (1.5,{2.25*0.4}) node[blue] [right] {$y=f_{\epsilon}(r)$};
    \node at (X0)[below,inner sep=5pt] {$1$};
    \node at (X1)[below,inner sep=5pt] {$1+\delta$};
    \node at (0.975, -0.27) [above][red] {$y=d_g(q_0, q_1)(r-1/2)+\epsilon/2$};
    \node at (-1.5,-0.1) [above left][blue] {$\frac{d_g(q_0,q_1)}{2} + \epsilon$};
    \node[draw,circle,fill=black,inner sep=1pt] at (Y){};
  \end{tikzpicture}
    \caption{Graph of the function $f_\epsilon$.}
    \label{Functionf}
\end{figure}
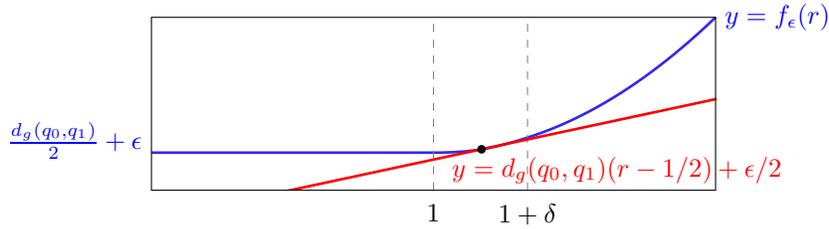

Let $H_s \in \mathcal{H}^{rel}_C$ be a path of Hamiltonians such that
$$
H_s (p) = \begin{cases}
\frac{d}{2} \|p\|^2_g, &\| p\|_g \geq 1 + \delta\\
H(p), & s\leq 0,\\
H_\epsilon(p), & s\geq 1.
\end{cases}
$$
Let $\Phi_{H_s}$ be a continuation map associated with the path $H_s$.

\begin{proposition}\label{Homotopy}
The following diagram commutes up to homotopy
\begin{equation*}
\begin{tikzpicture}
\matrix(a)[matrix of math nodes,
row sep=3em, column sep=2.5em,
text height=1.5ex, text depth=0.25ex]
{CW(T_{q_0}^*N, T_{q_1}^*N;H,J) & CW(T_{q_0}^*N, T_{q_1}^*N;H_\epsilon,J) \\
&CM(\mathcal{P}(q_0,q_1), \mathcal{L})\\};
\path[->]
(a-1-1) edge node[above]{$\Phi_{H_s}$}(a-1-2);
\path[->]
(a-1-1) edge node[below]{$\mathcal{F}_{H}$} (a-2-2);
\path[->]
(a-1-2) edge node[right]{$\mathcal{F}_{H_\epsilon}$} (a-2-2);
\end{tikzpicture}
\end{equation*}
\end{proposition}
\begin{proof}
For $\tau \in [0, +\infty)$ we set $H_{s,\tau}:=H_{s+\tau}.$ We also set $Z_-:=(-\infty,0] \times [0,1].$
For $x \in Crit(\mathcal{A}_{H})$ we define the moduli space $\mathcal{H}(x,\gamma)$ to be the set of pairs $(\tau,u)$ where $\tau \in [0, +\infty)$ and $u:Z_- \to T^*N$ satisfies the following conditions
\begin{equation*}
    \begin{cases}
    &\frac{\partial u}{\partial s}+J\left(\frac{\partial u}{\partial t}  - X_{H_{s,\tau}}\right)=0, \\
    &u((-\infty,0] \times \{i\}) \subset T_{q_i}^*N, \\
    &u_0(t):=u(0,t) \in N \subset T^*N,\\
    &\lim\limits_{s\to -\infty} u(s,t) = x(t), \quad u_0 \in W^s(\gamma).
    \end{cases}
\end{equation*}
From standard transversality results, it holds that $\mathcal{H}(x, \gamma)$ is a manifold for a generic choice of $J \in \mathcal{J}$. Applying a similar argument as in Lemma \ref{Escape}, we have that elements $u \in \mathcal{H}(x, \gamma)$ cannot escape to infinity. Hence, to prove that such a moduli space admits desired compactification, it is enough to show that the energy $E(u)$ of such $u$ is uniformly bounded in $\tau$. 
\begin{align*}
    E(u)&= \iint_{Z_-} \omega\left(\frac{\partial u}{\partial s}, J \frac{\partial u}{\partial s} \right) ds dt = \iint_{Z_-} \omega\left(\frac{\partial u}{\partial s},\frac{\partial u}{\partial t}  - X_{H_{s,\tau}}\right) ds dt \quad\\
    &= -\int u^* d\lambda + \iint_{Z_-} dH_{s,\tau}\left(\frac{\partial u}{\partial s} \right) ds dt.\\
    &= \int x^* \lambda - \int u_0^* \lambda + \iint_{Z_-} \frac{\partial}{\partial s} (H_{s,\tau} \circ u )ds dt - \iint_{Z_-} \frac{\partial H_{s,\tau}}{ \partial s} \circ u ds dt\\
    &\leq \mathcal{A}_{H}(x) - \mathcal{A}_{H_{\tau}}(u_0)+\iint_{Z_-} \bigg|\frac{\partial H_{s,\tau}}{ \partial s} \circ u \bigg| ds dt\\
    &= \mathcal{A}_{H}(x) - \mathcal{A}_{H_{\tau}}(u_0)+\int_0^1 \int_{-\tau}^{-\tau+1} \bigg|\frac{\partial H_{s,\tau}}{ \partial s} \circ u \bigg| ds dt ,\\
    &\leq \mathcal{A}_{H}(x) + \max_{s\in [0,1]}H_s\vert_N + \max_{s \in [0,1]} \left\| \frac{\partial H_s}{\partial s}\right\|_{C^0}. 
\end{align*}
Define a map $h:CW(T_{q_0}^*N, T_{q_1}^*N;H,J) \to CM(\mathcal{P}(q_0, q_1), \mathcal{L})$ by the equation
$$
h(x)=\sum_{\substack{\gamma \in Crit(\mathcal{L}) \\ dim \mathcal{H}(x, \gamma) = 0}} \#_2 \mathcal{H}(x,\gamma) \gamma.
$$
The components of the boundary of the compactified one-dimensional moduli space  $\mathcal{H}(y, \gamma)$ are of three different types. One is the fiber over $\tau=0$. This fiber can be identified with $\mathcal{M}^{\mathcal{F}_{H}}(y, \gamma)$ which is used to define the map $\mathcal{F}_{H}(y)$. The second type of component  of $\partial \mathcal{H}(y, \gamma)$ appears when $\tau \to +\infty$ and by standard compactness and gluing results is equal to
\begin{equation}
    \bigcup_{\substack{x \in Crit(\mathcal{A}_{H_\epsilon}) \\ dim \mathcal{M}(y, x ; H_s) = 0}} \mathcal{M}(y, x ; H_s) \times \mathcal{M}^{\mathcal{F}_{H_\epsilon}}(x, \gamma)\label{ContCompF}\end{equation}
 The count of the elements from the union (\ref{ContCompF}) gives $\mathcal{F}_{H_\epsilon} ( \Phi_{H_s} (y) )$. The third type of component comes from finite values of $\tau>0$, and it is equal to
\begin{align}
    \bigcup_{\substack{z \in Crit(\mathcal{A}_{H}) \\ dim \mathcal{M}(y, z) = 0}} &\mathcal{M}(y, z) \times \mathcal{H}(z, \gamma) \label{dCompH}\\
    \cup \bigcup _{\substack{\Tilde{\gamma} \in Crit(\mathcal{L}) \\ dim \mathcal{M}(\Tilde{\gamma}, \gamma) = 0}} &\mathcal{H}(y, \Tilde{\gamma}) \times \mathcal{M}(\Tilde{\gamma},\gamma). \label{HCompd}
\end{align}
 The union (\ref{dCompH}) induces $ h(dy)$ and the union (\ref{HCompd}) induces $\partial h(y)$. Hence we get
$$
\mathcal{F}_{H_\epsilon} \circ \Phi_{H_s} - \mathcal{F}_{H}= h \circ d - \partial \circ h.
$$
\end{proof}

\begin{proposition}\label{ThetaG}
If $\Bar{\gamma} \in \mathcal{P}(q_0, q_1)$ is the unique length minimizing geodesic from $q_0$ to $q_1$ and if $\Bar{x}$ is the unique Hamiltonian chord such that $\pi \circ \Bar{x} = \Bar{\gamma}$ then $\mathcal{F}_H([\Bar{x}])= [\pi \circ \Bar{x}].$
\end{proposition}
\begin{proof}
For $x\in Crti(\mathcal{A}_H)$, $\gamma \in Crit(\mathcal{L})$, and $u \in \mathcal{M}^{\Theta}(\gamma,x)$ we have $E(u)\leq \mathcal{L}(\gamma) - \mathcal{A}_H (x)$.
This energy inequality implies that $\mathcal{M}^{\Theta}(\pi \circ x,x)$ contains only the constant solution since $\mathcal{A}_H(x) = \mathcal{L}(\pi \circ x)$, hence $\#_2 \mathcal{M}^{\Theta}(\pi \circ x,x)=1.$ Further,  when $\mathcal{M}^{\Theta}(\Bar{\gamma},x)$ is zero dimensional and if $\mathcal{A}_H (x) \geq \mathcal{L}(\Bar{\gamma}) \text{ where } \Bar{\gamma} \neq \pi \circ x$ we have $\#_2 \mathcal{M}^{\Theta}(\gamma, \Bar{x}) = 0.$ Hence we get  for $\Bar{\gamma} = \pi \circ \Bar{x}$, which is a unique length minimizing geodesic, that $\Theta(\pi \circ \Bar{x}) = \Bar{x}.$ This is true since for $\pi \circ x \neq \Bar{\gamma}$ we have $\mathcal{A}_H(x)=\mathcal{L}(\pi \circ x) > \mathcal{L}(\Bar{\gamma})$.

We know from the Theorem \ref{Isomorphism} that $\Theta \circ \mathcal{F} = Id$ on homology. For every cycle $z \in CW(T_{q_0}^*N, T_{q_1}^*N;H,J)$, we get $\Theta ( \mathcal{F}_H ( [z]) ) = [z].$ The chord $\Bar{x}$ has the smallest action, since the differential drops the action we have that $\Bar{x}$ is a cycle. Because $\Theta$ is an isomorphism in homology, we have $\mathcal{F}_H([\Bar{x}])= [\pi \circ \Bar{x}].$

\end{proof}
The path $H_s$ can be chosen in a way such that the continuation map $\Phi_{H_s}$ satisfies $\Phi_{H_s}(\Bar{x}) = x_\epsilon.$
Indeed, take a smooth function $\chi: \mathbb{R} \to [0,1]$ such that $\chi'\geq 0$, $\chi(s)=0$ for $s \leq 0$, and $\chi(s)=1$ for $s \geq 1$, and set 
$$
H_s=(1-\chi)H + \chi H_{\epsilon}.
$$
From the proof of the Proposition \ref{Homotopy}, we have that energy of $u \in \mathcal{M}(\Bar{x}, y; H_s)$ satisfies
\begin{equation}\label{ContinuationEnergy}
    E(u)=\mathcal{A}_{H}(\Bar{x}) - \mathcal{A}_{H_\epsilon}(y) + \int\limits_{0}^{1}\int\limits_{0}^{1}  \chi'(s)(H -H_{\epsilon}) \circ u ds dt.
\end{equation}
The inequality $H\leq H_{\epsilon}$ implies $\mathcal{A}_{H_\epsilon}(y) \leq \mathcal{A}_{H}(\Bar{x})$. As we have seen, it follows from Remark \ref{ActionH=f} that if $y\neq x_{\epsilon}$ then $\mathcal{A}_{H_\epsilon}(y) \geq  d (1+\delta)^2 /2$ and $\mathcal{A}_{H}(\Bar{x}) = d / 2$, 
which means that for $y\neq \Bar{x}$ we have $\mathcal{M}(\Bar{x}, y; H_s) = \emptyset$. Equality $\Phi_{H_s}(\Bar{x})= x_\epsilon$ follows from $[\Bar{x}] \neq 0 \in HW(T_{q_0}^*N, T_{q_1}^*N;H,J)$.   

Now, let us show that $[\Bar{x}] \neq 0$.  From \cite[Section 2.4]{ASch05}  we have that the homology with $\mathbb{Z}_2$ coefficients of the sublevel set $\mathcal{L}^a:=\{ \gamma \in \mathcal{P}(q_0,q_1) \mid \mathcal{L}(\gamma)<a \},$  is isomorphic to the Morse homology $HM^a(\mathcal{P}(q_0, q_1))$. Here, $HM^a(\mathcal{P}(q_0, q_1))$ is the homology of the subcomplex $CM^a(\mathcal{P}(q_0, q_1)) \leq CM(\mathcal{P}(q_0, q_1))$ generated by the elements $\gamma \in crit(\mathcal{L})$ with $\mathcal{L}(\gamma) < a$. This isomorphism fits in the following commutative diagram

\begin{equation*}
\begin{tikzpicture}
\matrix(a)[matrix of math nodes,
row sep=3em, column sep=2.5em,
text height=1.5ex, text depth=0.25ex]
{HM^a(\mathcal{P}(q_0,q_1)) &  H(\mathcal{L}^a;\mathbb{Z}_2)\\
 HM(\mathcal{P}(q_0,q_1))& H(\mathcal{P}(q_0,q_1);\mathbb{Z}_2)\\};
\path[->]
(a-1-1) edge (a-1-2);
\path[->]
(a-1-1) edge node[left]{$i$} (a-2-1);
\path[->]
(a-1-2) edge node[right]{$i$} (a-2-2);
\path[->] (a-2-1) edge (a-2-2); 
\end{tikzpicture}
\end{equation*}
where the map represented by the left vertical is induced by the inclusion of the chain subcomplex $i:CM^a(\mathcal{P}(q_0,q_1)) \to CM(\mathcal{P}(q_0,q_1))$, and $i$ on the right is induced by the inclusion $i:\mathcal{L}^a \to \mathcal{P}(q_0,q_1)$. Now by taking $a:=d(1+\delta)$ we get that $[\Bar{\gamma}] \neq 0 \in HM^a(\mathcal{P}(q_0,q_1))$ corresponds to the point class of connected component containing $\Bar{\gamma}$ in $H(\mathcal{L}^a; \mathbb{Z}_2)$, since $CM^a(\mathcal{P}(q_0,q_1))$ is generated with $\Bar{\gamma}$. Since point class is preserved under the inclusion $i:H(\mathcal{L}^a; \mathbb{Z}_2) \to H(\mathcal{P}(q_0,q_1);\mathbb{Z}_2)$, we have $[\Bar{\gamma}]\neq 0 \in HM(\mathcal{P}(q_0,q_1))$. The map $\Theta$ from Theorem \ref{Isomorphism} is an isomorphism and $\Theta(\Bar{\gamma})=\Bar{x}$, hence it follows that the element $\Bar{x}$ is non-zero in homology.
 
\begin{corollary}\label{ExistenceJhol}
The moduli space $\mathcal{M}^{\mathcal{F}_{H_\epsilon}}(x_\epsilon,\pi\circ x_\epsilon)$ is non-empty. 
\end{corollary}
\begin{proof}
 From the Propositions \ref{Homotopy} and \ref{ThetaG} we have
 $$
 \mathcal{F}_{H_\epsilon} ([x_\epsilon]) = \mathcal{F}_{H_\epsilon} (\Phi_{H_s} ([\Bar{x}])) =\mathcal{F}_H([\Bar{x}])= [\pi \circ \Bar{x}]=[\pi \circ x_\epsilon]\neq 0.
 $$

\end{proof} 
\begin{lemma} \label{MainEnergy}
For $u \in \mathcal{M}^{\mathcal{F}_{H_\epsilon}}(x_\epsilon,\pi\circ x_\epsilon)$ we have $E(u)= d_g (q_0,q_1)+\epsilon/2.$
\end{lemma}
\begin{proof}
Intuitively, from Figure (\ref{J-hol}) it follows that area of $u$ is roughly $d_g(q_0,q_1).$ More precisely, from Remark (\ref{ActionH=f}) we have $\mathcal{A}_{H_\epsilon}(x_\epsilon)=r_{\epsilon}f_\epsilon'(r_{\epsilon})-f_\epsilon(r_{\epsilon}).$ The choice of the function $f_\epsilon$ was such that $r_{\epsilon}f_\epsilon'(r_{\epsilon})-f_\epsilon(r_{\epsilon})=d_g (q_0,q_1) / 2 - \epsilon/2$.
From the equation (\ref{EnergyF}) we have
$$
E(u)=\mathcal{A}_{H_\epsilon}(x_\epsilon) + H\vert_N = d_g (q_0,q_1)+\epsilon /2,
$$
since $H\vert_N= d_g (q_0,q_1) / 2 + \epsilon$.
\end{proof}

\subsection{Proof of Proposition \ref{MainProp}} \label{MainPropProof}

The following theorem from Riemannian geometry guarantees that generically, every two points on $(N,g)$ are connected by a unique length minimizing geodesic. 

\begin{proposition}\cite[Theorem 18.1.]{Mi63}\cite[III. Lemma 4.4]{Sa96}  \label{LebZero}
Let $(N,g)$ be a closed Riemannian manifold. For $p \in N$ define the set
$\mathcal{I}_p:= \{ q \in N \mid \text{There is a unique length minimizing geodesic from } p  \text{ to } q \}.$
Set $\mathcal{I}_p$ is open and dense. Moreover, its complement $\mathcal{C}_p=N\setminus \mathcal{I}_p$ has Lebesgue measure zero.
\end{proposition}

As a corollary, we can construct a small Hamiltonian perturbation in $D^*N$, which preserves the zero section $N$, fixes the image of $e_0:B^{2n}(r_0) \to D^*N$ restricted to a slightly smaller ball, and moves the image of $e_1:B^{2n}(r_1) \to D^*N$ so that centers $q_i=e_i(0)$ are connected by a unique length minimizing geodesic.
\begin{corollary} \label{UniqueGeod}
Let $e_i:B^{2n}(r_i) \to D^*{N}, \quad i\in \{0,1\}$ be two relative, disjoint, symplectic embeddings. For every $\epsilon>0$ there exist a Hamiltonian $G:T^*N \to \mathbb{R}$ such that 
\begin{itemize}
    \item[1.)] $\phi^1_G \vert_{e_0(B^{2n}(r_0-\epsilon) )} = Id$,
    \item[2.)] $e_0(B^{2n}(r_0-\epsilon))\cap \phi^1_G(e_1(B^{2n}(r_1-\epsilon) )) = \emptyset$,
    \item[3.)] there is a unique length minimizing geodesic from $q_0:=  e_0 (0) $ to $q_1:= \phi_G^1 (e_1 (0) )$,
    \item[4.)] $\phi^1_H(T_{q_0}^* N) \pitchfork T^*_{q_1}N$, where $H(p)=d_g(q_0, p_1) \|p\|^2_g /2$, and $p_1=e_1(0)$,
    \item[5.)] If $e_1^{-1}(T^*_{p_1}N)=\{0\} \times B^n(r_1),$ where $p_1=e_1(0)$, then $G$ can be chosen so that
    $$(\phi^1_G \circ e_1)^{-1} (T^*_{q_1}(N)) \cap B^{2n}(r_1-\epsilon) = \{0\} \times B^n(r_1-\epsilon).$$
\end{itemize} 
\end{corollary} 

\begin{proof}
Fix $\epsilon>0$ such that $B(p_1;2\epsilon):=\{q \in N \mid d_g (q,p_1)<2\epsilon\} \subset e_1(B^{2n}(r_1 - \epsilon)) $. Fix a smooth function $\rho:N \to [0,1]$ such that $\rho  = 1$  on the set $B(p_1;\epsilon)$ and $\rho  = 0$ on the complement of $B(p_1;2\epsilon)$. Now, for any $0<\delta < \epsilon$ 
 we have that $U_\delta:=\mathcal{I}_{q_0}\cap B(p_1, \delta) \neq \emptyset$ by Proposition \ref{LebZero}. Fix any point $q^\delta_1 \in U_\delta$, such that $d_g(q_0, q^1_\delta) < d_g(q_0,p_1)$, and let $\gamma_\delta:[0,1] \to N$ be the unique length minimizing geodesic from $p_1$ to $q^\delta_1$. Such a geodesic exists since $q^\delta_1$ is in the normal neighborhood $B(p_1, \delta)$  of $p_1$. By Sard's theorem, we can assume that $q^\delta_1$ is such that $\phi^1_H(T_{q_0}^* N) \pitchfork T^*_{q^\delta_1} N$.
 
 Choose any vector field $X_\delta$ on $N$ which extends $\gamma_{\delta} '$ and such that $\|X_\delta \| \leq 2\delta$. Now, we define a Hamiltonian $G:T^*N \to \mathbb{R}$ by the following equation $G_\delta(p)=\rho(\pi(p)) p(X_\delta).$ For $\delta$ small enough, it is easy to see that $\phi^1_{G_\delta}$ satisfies
 \begin{itemize}
     \item $\phi^1_{G_\delta}(p_1) = q^\delta_1$ and $\phi^1_{G_\delta}(T^*_{p_1}N) = T^*_{q^\delta_1}N $,
     \item $\phi^1_{G_\delta}(N) = N$,
     \item $\phi^1_{G_\delta} (e_1(B^{2n}(r_1-\epsilon))) \subset e_1(B^{2n}(r_1))$.
 \end{itemize}
Then, Hamiltonian $G:= G_\delta$ satisfies the assumptions of the Corollary.
\end{proof}

\begin{remark}
It is crucial that the set $\mathcal{I}_{q_0}$ is both open and dense in order to have a non-empty intersection with regular values from $B(p_1, \delta)$ of the map $\pi: \phi_H^1(T_{q_0}^* N) \to N$  for any $\delta$. Here $\pi$ is the restriction of $\pi:T^*N \to N$ to $\phi^1_H(T_{q_0}^* N)$.
\end{remark}

We are in a position to prove the Proposition \ref{MainProp}. 
\begin{proof}[{{Proof of the Proposition \ref{MainProp}}}]
Fix a symplectic embedding 
$$e:B^{2n}(r_0)\sqcup B^{2n}(r_1) \to D^*{N} ,$$ 
relative to $N$ and such that $e^{-1}(D_{q_0}^*N \sqcup D_{q_1}^*N)=\{0\}\times B^{n}(r_0) \sqcup \{0\}\times B^{n}(r_1)$. Set $e_i:= e \vert_{B^{2n}(r_i)}$.  By replacing $e$ with $\phi_G^1 \circ e$, Corollary \ref{UniqueGeod}  implies that there is a unique length minimizing geodesic $\Bar{\gamma}$ from $q_0=e_0(0)$ to $q_1=e_1(0)$.  In Corollary \ref{UniqueGeod}, we have achieved  $\phi^1_H(T_{q_0}^* N) \pitchfork T^*_{q_1}N$ for Hamiltonian $H(p)=d_g(q_0, p_1) \|p\|^2_g /2$, where $p_1$ was the center of the initial embedding $e_1$. In Section \ref{Technical results} we could have worked with $H= C \|p\|_g^2$ instead of $H= d_g(q_0, q_1)/2 \|p\|_g^2$, where  $d_g(q_0, q_1)\leq 2C < d_g(q_0, q_1)(1+\delta)$,  since $q_1$ was chosen from arbitrarily small neighborhood of $p_1$  one can put $C=d_g(q_0, p_1)/2$. Because the properties of corresponding function $f_\epsilon$ do not change on the region we are interested we can actually assume that $C=d_g(q_0, q_1)/2$.  Let $J\in \mathcal{J}$ be such that
\begin{equation}\label{ACSonImageB}
    J\vert_{e_i(B^{2n}(r_i))} = {e_i}_* J_0.
\end{equation}
 Recall that $x_\epsilon$ is the unique Hamiltonian chord of $H_\epsilon(p)=f_\epsilon(\|p\|^2)$ which corresponds to the geodesic $\Bar{\gamma} = \pi \circ x_\epsilon$. From the Corollary \ref{ExistenceJhol} and Lemma \ref{MainEnergy} we have $u \in \mathcal{M}^{\mathcal{F}_{H_\epsilon}}(x_\epsilon, \Bar{\gamma})$ (Figure (\ref{J-hol})) such that
\begin{equation}\label{MainEnergyEq}
    E(u)\leq d_g(q_0, q_1) + \frac{\epsilon}{2}.
\end{equation}
Transversality of the moduli space $\mathcal{M}^{\mathcal{F}_{H_\epsilon}}(x_\epsilon, \Bar{\gamma})$ can be achieved for an almost complex structure which satisfies equation (\ref{ACSonImageB}) by perturbing $J$ outside of $D^*N$ (\cite[Lemma 3.4.4]{MS12}).
Since $H_\epsilon$ is constant on $D^*N$, map $u$ is $J$-holomorphic on $u^{-1}(D^*N) \subset Z_-$. The rest of the proof is analogous to the idea from \cite{BM14}. Let us set $\Bar{\Sigma}_i$ to be a connected component of $(i,0)$ in $u^{-1}(e_i(B^{2n}(r_i)))$, and define maps 
$$\Bar{v}_i:= e_i^{-1} \circ u : \Bar{\Sigma}_i \to B^{2n}(r_i).$$ 
Set
\begin{align*}
    M_i:&=\sup \{r >0 \mid \forall s \in [0,1] \quad u(rs,i) \in e_i(\{0\} \times B^{n}(r_i))\},\\
    a_i:&=\sup \{r >0 \mid \forall t \in [0,1] \quad u(i,|i - rt|) \in e_i(B^{n}(r_i)\times \{0\}),
\end{align*}
and $\partial_v \Bar{\Sigma}_i:= [0,M_i]\times \{i\}, \quad \partial_h \Bar{\Sigma}_i:=\{i\} \times[0,a_i].$ From the properties of $u$, we have that $\partial \Bar{\Sigma}_i =  \partial_v \Bar{\Sigma}_i \cup \partial_h \Bar{\Sigma}_i \cup C$ where $\Bar{v}_i (C) \subset \partial B^{2n}(r_i)$. Also $\Bar{v}_i(i,0)=0\in B^{2n}(r_i)$, and
\begin{align*}
    &\Bar{v}_i(\partial_h \Bar{\Sigma}_i) \subset   B^{n}(r_i) \times \{0\},& &{\Bar{v}_i(i,a_i) \in \partial B^{2n}(r_i)},\\
    &\Bar{v}_i(\partial_v \Bar{\Sigma}_i) \subset \{0\} \times B^{n}(r_i),& &{\Bar{v}_i(M_i,i) \in \partial B^{2n}(r_i)}.
\end{align*}
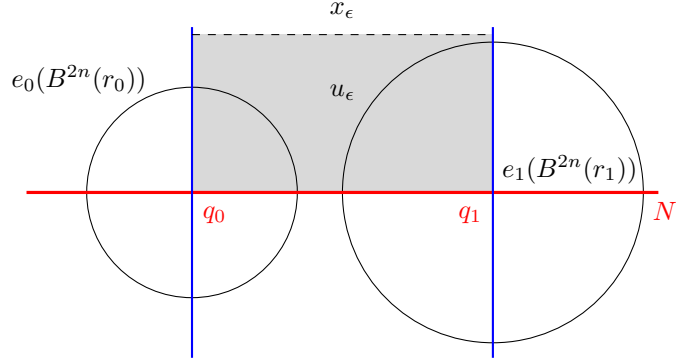
\begin{figure}[ht]
    \centering
    \begin{tikzpicture}
\filldraw[gray!30] (-2,0) rectangle (2,2.1);
\draw (-2,0) circle (1.4);
\draw (2,0) circle (2);
\node[red] at (-1.7,-0.3) {$q_0$};
\node[red] at (1.7,-0.3) {$q_1$};
\node at (-3.5, 1.5) {$e_0(B^{2n} (r_0))$};
\node at (2,0) [above right] {$e_1(B^{2n} (r_1))$};
\node[red] at (4,-0) [below right] {$N$};
\node at (0,2.2) [above] {$x_\epsilon$};
\node at (0,1.1) [above] {$u_\epsilon$};
\draw[red, very thick] (-4.2,0) -- (4.2,0);
\draw[blue, thick] (-2,-2.2) -- (-2,2.2);
\draw[blue, thick] (2,-2.2) -- (2,2.2);
\draw[dashed] (-2,2.1) -- (2,2.1);
\end{tikzpicture}
    \caption{Image of $u \in \mathcal{M}^{\mathcal{F}_{H_\epsilon}}(x_\epsilon, \Bar{\gamma})$.}
    \label{J-hol}
\end{figure}
We chose an almost complex structure $J$ to be pushforward of the standard complex structure $J_0$ on the images $e_i(B^{2n}(r_i))$. Hence, the maps $\Bar{v}_i$ are holomorphic. Since $\partial_h \Bar{\Sigma}_i$ is mapped to the real part of the ball $B^{2n}(r_i)$, and $\partial_v \Bar{\Sigma}_i$ is mapped to the imaginary part we can apply Schwartz reflection two times to get the maps $v_i:\Sigma_i \to B^{2n}(r_i).$ From the properties of $\Bar{v}_i$ we have that $v_i(\partial \Sigma_i) \subset \partial B^{2n}(r_i)$ and $v_i(0)=0$. Since $v_i$ is a holomorphic map, its image is a minimal surface. Applying monotonicity property of minimal surfaces to $v_i$ (\cite{AO75, MS12}) we get
\begin{equation}\label{Isoperimetric}
    \pi r_i^2 \leq Area(v_i)=E(v_i).
\end{equation}
Since $v_i$ is obtained from $\Bar{v}_i$ by applying Schwartz reflection twice, we have $E(v_i)=4E(\Bar{v_i})$. From the equations (\ref{MainEnergyEq}), (\ref{Isoperimetric}) and $E(v_i)=4E(\Bar{v_i})$  we get
\begin{equation}\label{Gr<diam}
     \pi r_0^2+\pi r_1^2 \leq E(v_0) + E(v_1) = 4 (E(\Bar{v}_0)+E(\Bar{v}_1))\leq 4 E(u) = 4 d_g(q_0, q_1) + 2\epsilon 
\end{equation} Here the second inequality holds since
\begin{align*}
    E(\Bar{v}_0)+E(\Bar{v}_1)& = \iint_{\Bar{\Sigma}_0} \Bar{v}_0^*\omega_{st} + \iint_{\Bar{\Sigma}_1} \Bar{v}_1^*\omega_{st}= \iint_{\Bar{\Sigma}_0 \cup \Bar{\Sigma}_1} u^*\omega \\
    &= \iint_{\Bar{\Sigma}_0 \cup \Bar{\Sigma}_1} \omega\left(\frac{\partial u}{\partial s}, J \frac{\partial u}{\partial s} \right) ds dt \leq \iint_{Z_-} \omega\left(\frac{\partial u}{\partial s}, J \frac{\partial u}{\partial s} \right) ds dt = E(u).
\end{align*}
 Since the inequality (\ref{Gr<diam}) holds for every $\epsilon>0$ we get
$\pi r_0^2+\pi r_1^2 \leq 4 d_g(q_0, q_1).$
\end{proof}
 \begin{remark}\label{packing in r-bundle}
 By a rescaling argument, one can show an analogous result for the $r$-disc cotangent bundle $D^*_r N = \{ p \in T^*N \mid \| p \|_g < r \}$. In other words, if we have embeddings $e_i:B^{2n}(r_i) \to D^*_r N$ with real and imaginary part constraints as in  Proposition \ref{MainProp} the bound is $\pi r_0^2 + \pi r_1^2 \leq 4 r d_g(q_0, q_1).$
 \end{remark}

\subsection{Proof of Theorem \ref{Main}}\label{ProofMain}
Now we properly define the function $\rho_\mathcal{W}$ from the introduction. Fix a closed Lagrangian submanifold $L$ of a symplectic manifold $(M,\omega)$ and let $\mathcal{W}_L$ be a Weinstein neighborhood of $L$, i.e. $\mathcal{W}_L$ is symplectomorphic to an open neighborhood $\mathcal{V}$ of the zero section $\mathcal{O}_L$ in $(T^*L, -d\lambda)$, sending $L$ to $\mathcal{O}_L$.  If $\mathcal{W}_L$ is a bounded\footnote{Meaning that $\overline{\mathcal{W}}_L$ is a compact subset of $M$} Weinstein neighborhood we define a distance-like function $\rho_{\mathcal{W}_L}: L \times L \rightarrow [0,\infty)$. 
\begin{definition} \label{MainDefinition}
If $q_0 \neq q_1$ then
\begin{equation*} 
    \rho_{\mathcal{W}_L}(q_0, q_1):= \sup \left\{ \frac{\pi r^2}{2}  \left. \begin{aligned} \text{}\\ \text{}\\ \text{}
 \end{aligned}\right\vert 
        \begin{aligned}
           &\exists e: B_0^{2n}(r) \sqcup B_1^{2n}(r) \rightarrow \mathcal{W}_L,  \quad e^* \omega = \omega_{st}, \\
           &e^{-1}(L) = B_0^{n}(r)  \times \{0\} \sqcup B_1^{n}(r)  \times \{0\},\quad\\
           &e^{-1}\left(\bigsqcup \mathcal{W}_{q_i}\right) = \bigsqcup \{0\} \times B_i^{n}(r),
        \end{aligned} \right\}. 
\end{equation*}
where $\mathcal{W}_{q}$ is the image of the fiber $T^*_q L\cap \mathcal{V}$.  We set $\rho_{\mathcal{W}_L}(q,q)=0$.
\end{definition}
In the following lemma, we construct explicit symplectic embeddings to Lagrangian bi-disc $P_L(a,b)$, which is needed to obtain a sufficiently nice lower bound for $\rho_{\mathcal{W}_L}$.
\begin{lemma}\label{Sharpness for bi-disc}
There exists a relative symplectic embedding $$e:B^{2n}\left(2\sqrt{\frac{ab}{\pi}}\right) \to P_L(a,b).$$
\end{lemma}

\begin{proof}
Let us first find a symplectic embedding $e:B^2\left(2\sqrt{ab/\pi}\right) \to (-ab, ab) \times (-1, 1)$ which has the form $e(q,p) = \left(f(q), \frac{1}{f'(q)}p\right)$. Setting $$f(q)=\frac{2ab}{\pi}\mathrm{arcsin}\left(\sqrt{\frac{\pi}{4 ab}}q\right) + \frac{q}{2} \sqrt{\frac{4ab}{\pi} - q^2},$$
we get the desired symplectic embedding $e$. Indeed, $f$ is odd and increasing, and $f\left( 2\sqrt{\frac{ab}{\pi}}\right)=ab$. Also, $f'(q)=\sqrt{\frac{4ab}{\pi} - q^2}$, hence for $p^2 < \frac{4ab}{\pi} - q^2$ we have $\left|\frac{1}{f'(q)}p \right| < 1.$
For higher dimensions we set $\varphi(q):= \frac{f(\|q\|)}{\|q\|}q$. Since $f$ is odd and analytic, $\varphi$ is smooth. It is easy to see that
$$
D\varphi(q)h = \left(f'(\|q\|)-\frac{f(\|q\|)}{\|q\|}\right) \frac{\langle q , h \rangle}{\|q\|} \frac{q}{\|q\|} + \frac{f(\|q\|)}{\|q\|}h,
$$
and
\begin{equation}\label{LowerBoundonDFi}
     \|D\varphi(q)h\| \geq \frac{f(\|q\|)}{\|q\|} \|h\| - \left\vert f'(\|q\|)-\frac{f(\|q\|)}{\|q\|} \right\vert \frac{|\langle q , h \rangle|}{\|q\|} \geq f'(\|q\|) \|h\|.
\end{equation}
   The last inequality holds since $\frac{f(t)}{t} > f'(t)$ for $t>0$. Since $f'(t)>0$ for $t\in \left[0, 2 \sqrt{\frac{ab}{\pi}}\right)$ we get that $D\varphi(q)$ is invertible. Now define symplectic embedding $e:B^{2n}\left(2\sqrt{ab/\pi}\right) \to P_L(ab,1)$ as $e(q,p):=\left(\varphi(q), (D\varphi(q)^{-1})^T p\right).$ It follows from (\ref{LowerBoundonDFi}) that $Im(e) \subset P_L(ab,1)$.
This ends the proof since $P_L(a,b)$ is symplectomorphic to $P_L(ab,1)$. To see that such embedding $e$ satisfies the imaginary part condition, it is enough to note that $e^* \lambda_{st} = \lambda_{st}$ where $\lambda_{st} = \sum p_i dq_i$.
\end{proof}

Using Lemma \ref{Sharpness for bi-disc} and Proposition \ref{Main}, we can prove that $\rho_{\mathcal{W}_L}$ is equivalent to a distance $d_g$ coming from a Riemannian metric $g$ on $L$. 
\begin{proposition}
Let $g$ be some Riemannian metric on a closed manifold $L$. There are $C_i(\mathcal{W}_L,g)>0$ such that for any $q_0, q_1 \in L$ we have $C_1 d_g(q_0, q_1) \leq \rho_{\mathcal{W}_L}(q_0, q_1) \leq C_2 d_g(q_0, q_1).$
\end{proposition}

\begin{proof}
Set $r_{max}:= \inf \{ r \mid \mathcal{W}_L \subset D^*_r L\}$ and $r_{min}:= \sup \{ r \mid  D^*_r L \subset  \mathcal{W}_L \}$.
It follows from Remark \ref{packing in r-bundle} that $\rho_{\mathcal{W}_L} \leq r_{max} d_g,$
so we can set $C_2:= r_{max}$. For the lower bound, let $\rho_{inj}$ be the injectivity radius of Riemannian metric $g$, and set
$$
A:= \min_{q \in N} \min_{\|p\| \leq \frac{\rho_{inj}}{2}}  \frac{1}{\|(D exp_q (p) ^{-1})^T \|}.
$$
Since $exp$ is a radial isometry, and it is a diffeomorphism for $\|p \| \leq \frac{\rho_{inj}}{2}$, from the compactness of $L$ we get $0< A \leq 1.$

First, in the case when $d_g(q_0, q_1) \leq \rho_{inj}$, we can explicitly construct symplectic embeddings $$\psi_i:P_L\left(d_g(q_0, q_1)/2,A r_{min}\right) \to D^*_{r_{min}}L,$$ such that $\psi_i(0,0)=q_i$ and $\psi^*_i \lambda = \lambda_{st}$. Set $\varphi_i(q):=exp_{q_i}(q)$ and $\psi_i(q,p):=(\varphi_i(q), (d\varphi_i(q)^{-1})^* p).$ 
Here we used unitary linear identification $T_{q_i}L \times T_{q_i}^* L \cong \mathbb{R}^{2n}$. The constant $A$ is chosen in a way that image of $\psi_i$ remains inside $r_{min}$ disc-cotangent bundle $D^*_{r_{min}} L$. It is easy to see that $\psi_i^*\lambda = \lambda_{st}$ and $\psi_i(0,0)=q_i$. From Lemma \ref{Sharpness for bi-disc}, we have a symplectic embedding $$e:B^{2n}(r) \to P_L\left(d_g(q_0, q_1)/2,A r_{min}\right),$$ where $r$ is such that $\pi r^2 = 2 A r_{min} d_g(q_0,q_1)$. Looking at the compositions $\psi_i \circ e$ we get two symplectic embeddings of the ball of capacity $2 A r_{min} d_g(q_0,q_1)$, centered at $q_0$ and $q_1$, satisfying constraints on the real and imaginary parts. Hence we have $$ \rho_\mathcal{W}(q_0, q_1) \geq A r_{min} d_g(q_0,q_1).$$
Now, if $d_g(q_0,q_1) > \rho_{inj}$, we know that $\rho_\mathcal{W}(q_0,q_1) \geq A r_{min} \rho_{inj}$. By a simple estimate, we get
$$
\rho_\mathcal{W}(q_0,q_1) \geq A r_{min} \rho_{inj} \frac{diam (L)}{diam (L)} \geq \frac{A r_{min}\rho_{inj}} {diam(L)} d_g(q_0, q_1),
$$
since $\frac{A r_{min}\rho_{inj}} {diam(L)} \leq A r_{min} $ we can set $C_1:= \frac{A r_{min}\rho_{inj}} {diam(L)}.$
\end{proof}

Before we start proving the main theorem, we need one more technical lemma. This lemma gives a better lower bound for $\rho_\mathcal{W}(q_0,q_1)$ in the case $\mathcal{W}=D^*N$ and when points $q_0, q_1$ are close enough.

\begin{lemma}\label{Local distance}
There exists $\delta_0 > 0$ such that for all $q \in N$ and all $d < \delta_0$ we have a symplectic embedding
$$
\psi:P_L\left(d, \sqrt{\frac{1}{1+d}}\right) \to D^*N,
$$
such that $\psi(0,0)=q$ and $\psi^* \lambda = \lambda_{st}$.
\end{lemma}
\begin{proof}
From \cite[Lemma 5.5.7.]{Pe16} we have that in exponential coordinates around any $q_0\in N$ coeficients of $g$ satisfy $g_{ij}=\delta_{ij} + O(r^2)$ where $r$ is distance from $q_0$. From compactness of $N$ one can show that there is $\delta_0>0$ such that for any $q_0 \in N$, for $q \in B(q_0,\delta_0)$ we have $\|G^{-1}(q) - I\| \leq \|q\|,$ where $G$ is matrix associated to coordinate components $g_{ij}$ of metric $g$ and $I$ is identity matrix. Our map $\psi:P_L\left(d, \sqrt{\frac{1}{1+d}}\right) \to T^*N$ is of the form $\psi(q,p)=(\varphi(q), (d\varphi(q)^{-1})^* p),$
where $\varphi(q)=exp_{q_0}(q)$. Take $\|p\| \leq \sqrt{\frac{1}{1+d}}$ and $\|q\|\leq d < \delta_0$, we have
$$
\|(d\varphi(q)^{-1})^* p\|^2 = p^T G^{-1}(q) p =  |p^T(G^{-1}-I)p + p^Tp| \leq \|p\|^2(\|q\| + 1)<1,
$$
this estimate proves that $Im(\psi) \subset D^*N$.

\end{proof}

\begin{proof}[{{Proof of the Theorem \ref{Main}}}]
Let $L:=N$ be the zero section in $T^*N$ and $\mathcal{W}=D^*N$ the unit-disc bundle. We say that partition  $\mathcal{P}$ is admissible if $d_g(\gamma(t_i), \gamma(t_{i+1}) )<\delta_0$,  for all $i \in \{ 1,...,k\}$,
where $\delta_0$ is from Lemma \ref{Local distance}. Set $\gamma_i:=\gamma(t_i)$. It is easy to see that
$$
L_{\rho_\mathcal{W}}(\gamma)= \sup \left\{ \sum_{1\leq i \leq k}\rho_\mathcal{W}(\gamma_i , \gamma_{i+1} ) \mid \mathcal{P} \text{ is admissible }\right\}.
$$
By the choice of $\delta_0$, it follows from Lemmas \ref{Sharpness for bi-disc} and \ref{Local distance} that
$$
d_g(\gamma_i , \gamma_{i+1} ) \sqrt{\frac{1}{1+d_g(\gamma_i , \gamma_{i+1} )}} \leq \rho_\mathcal{W}(\gamma_i , \gamma_{i+1}  ).
$$
On the other hand, from Proposition \ref{MainProp}, we have $\rho_\mathcal{W}(\gamma_i , \gamma_{i+1}  ) \leq d_g(\gamma_i , \gamma_{i+1} ).$ When we take a sum over all $i$ we get
\begin{equation}\label{Partial sum}
    \sum d_g(\gamma_i , \gamma_{i+1} ) \sqrt{\frac{1}{1+d_g(\gamma_i , \gamma_{i+1} )}} \leq \sum \rho_\mathcal{W}(\gamma_i , \gamma_{i+1} ) ) \leq \sum  d_g(\gamma_i , \gamma_{i+1} ).
\end{equation}
It is a standard fact from Riemannian geometry that $$\int_a^b \|\gamma'(t)\|_g dt = \lim_{\lambda(\mathcal{P}) \to 0} \sum  d_g(\gamma_i , \gamma_{i+1} ),$$ where $\lambda(\mathcal{P}) = \max\limits_i (t_{i+1} - t_i)$. Since $\gamma$ is uniformly continuous, we have that $\lambda(\mathcal{P}) \to 0$ implies $\delta(\mathcal{P}):=\max_i d_g(\gamma_i, \gamma_{i+1}) \to 0.$
Inserting $\delta(\mathcal{P}) \geq d_g(\gamma_i, \gamma_{i+1})$ in the equation (\ref{Partial sum}) we get $$\sqrt{\frac{1}{1+\delta(\mathcal{P})}}\sum d_g(\gamma_i , \gamma_{i+1} )  \leq \sum \rho_\mathcal{W}(\gamma_i , \gamma_{i+1} ) ) \leq \sum  d_g(\gamma_i , \gamma_{i+1} ).$$
Now, letting $\lambda(\mathcal{P}) \to 0$ we get from the standard squeeze theorem for limits that 
$$L_{\rho_{\mathcal{W}}}(\gamma) = \int_a^b \|\gamma'(t)\|_g dt.$$
\end{proof}

\end{document}